\documentclass[pdflatex,sn-mathphys-num]{sn-jnl}% Math and Physical Sciences Numbered Reference Style
\usepackage{amsmath}
\usepackage{amssymb}
\usepackage{enumerate}
\usepackage{algorithm,algorithmic}
\usepackage{caption}
\usepackage{array}
\usepackage{pgfplots, pgfplotstable, booktabs, colortbl, array}
\usepackage{multirow}

\usepackage{float}
\floatstyle{plaintop}
\restylefloat{table}

\usepackage{todonotes}

\newcommand{\ignore}[1]{}

\usepackage{amsfonts}
\usepackage{amssymb}

\newcommand{\norm}[1]{\|#1\|}
\newcommand{\abs}[1]{|#1|}

\newcommand{\uu}{{\bf u}}


\newtheorem{theorem}{Theorem}[section]
\newtheorem{lemma}[theorem]{Lemma}

\newtheorem{assumption}[theorem]{Assumption}

\theoremstyle{definition}
\newtheorem{definition}[theorem]{Definition}

\theoremstyle{remark}

\theoremstyle{thmstyleone}%

\theoremstyle{thmstyletwo}%

\theoremstyle{thmstylethree}%

\begin{document}

\title[RCLUPPr]{Analysis of RCLUPPr: stability, robustness and acceleration}


\author[1]{\fnm{Haoran} \sur{Guan}\textsuperscript{\textdagger}}\email{guan.haoran@huawei.com}

\author[1,2]{\fnm{Zhenyu} \sur{Zou}\textsuperscript{\textdagger}}\email{zou.zhenyu1@huawei.com}

\author*[1]{\fnm{Yufeng} \sur{Wei}}\email{wei.yufeng@huawei.com}

\author[1]{\fnm{Yipei} \sur{Chen}}\email{chen.yipei1@huawei.com}

\author[1]{\fnm{Peiting} \sur{You}}\email{youpeiting@huawei.com}

\author[1]{\fnm{Yuwei} 
\sur{Fan}}\email{fanyuwei2@huawei.com}

\affil*[1]{\orgdiv{Theory Lab}, \orgname{Huawei Leibniz Research Center}, \orgaddress{\street{Sha Tin}, \city{New Territories}, \postcode{999077}, \state{Hong Kong SAR}, \country{China}}}

\affil[2]{\orgdiv{Department of Mechanical and Automation Engineering}, \orgname{The Chinese University of Hong Kong}, \orgaddress{\street{Sha Tin}, \city{New Territories}, \postcode{999077}, \state{Hong Kong SAR}, \country{China}}}


\abstract{In this work, we present a comprehensive rounding error analysis for RCLUPPr, a recently developed randomized CholeskyQR-type algorithm introduced in \cite{RCLUPP}. This method performs LU factorization with partial pivoting (LUPP faztorization) directly on a full-rank tall-skinny matrix $X \in \mathbb{R}^{m \times n}$. Unlike RCLUPP in \cite{RCLUPP}, which applies matrix sketching prior to LUPP factorization, RCLUPPr deploys LUPP as an initial preconditioning step to significantly mitigate the propagation of error significantly. Our rigorous analysis demonstrates that RCLUPPr exhibits superior robustness when applied to the ill-conditioned matrices compared to other CholeskyQR-type algorithms in the high precision. Furthermore, we develop some practical strategies in the implementation to accelerate RCLUPPr. Extensive numerical experiments on the real-world problems validate our theoretical findings, showcasing the robustness and the efficiency of RCLUPPr across the single, double and the mixed-precision frameworks.}

\keywords{LU factorization, QR factorization, matrix sketching, Computational mathematics}



\maketitle

\section{Introduction}
\subsection{Developments of CholeskyQR-type algorithms}
In recent years, CholeskyQR has raised much concern as an algorithm for QR factorization of a full-rank tall-skinny matrix $X\in\mathbb{R}^{m\times n}$, see Algorithm~\ref{alg:cholqr}. Here, $R$ is the $R$-factor computed from Cholesky factorization of $G$. For better accuracy in orthogonality, CholeskyQR2 \cite{2014,error} applies CholeskyQR twice. For $X$ with large $\kappa_{2}(X)$, typically with $\kappa_2(X)\geq 10^8$ in double precision, more robust variants have been developed. Such algorithms include Shifted CholeskyQR3 (SCholeskyQR3) \cite{Columns,Shifted}, which introduces a shifted item $sI$ in the first Cholesky factorization to avoid breakdown, and LU-CholeskyQR2 (LUC2) \cite{LUChol}, which combines LU factorization with partial pivoting (LUPP) with CholeskyQR.

\begin{algorithm}
\caption{$[Q,R]=\mbox{CholeskyQR}(X)$ for $X \in \mathbb{R}^{m\times n}$}
\label{alg:cholqr}
\begin{algorithmic}[1]
\STATE $G=X^{\top}X,$
\STATE $R^{\top}R=G,$
\STATE $Q=XR^{-1}.$
\end{algorithmic}
\end{algorithm}%

Driven by advances in randomized linear algebra \cite{halko2011,2020}, matrix sketching has become an effective tool for CholeskyQR-type algorithms \cite{Randomized,Novel,Householder}. By projecting a tall-skinny matrix into a lower-dimensional space, one can compute a preconditioner using numerically stable methods, such as the thin Householder QR or Givens rotations \cite{Higham}. This strategy mitigates the potential breakdown of standard Cholesky factorization at a modest additional cost.

\subsection{RCLUPP and RCLUPPr}
Motivated by randomized CholeskyQR (RCholeskyQR), we propose RCLUPP \cite{RCLUPP}, an improved variant of LUC2 that combines the thin Householder QR with LUPP factorization. As shown in Algorithm~\ref{alg:RCLUPP}, RCLUPP applies matrix sketching directly to the input matrix $X$, thereby avoiding the numerical breakdown of LUC2 caused by a large $\kappa_2(L)$. Here, $\Omega$ is the matrix for sketching. $H$ and $Y_{0}$ denote the $Q$-factor and the $R$-factor from the thin HouseholderQR. $Y_{1}$ is the $R$-factor from the preconditioning step. $Z$ is the $R$-factor when taking CholeskyQR2 for $W$. Compared with RCholeskyQR, RCLUPP delivers better orthogonality with comparable, and sometimes lower, computational cost \cite[Section~6]{RCLUPP}. 

\begin{algorithm}
\caption{$[Q,R]=\mbox{RCLUPP}(X)$ with $X \in \mathbb{R}^{m\times n}$}
\label{alg:RCLUPP}
\begin{algorithmic}[1]
\STATE $A=\Omega X,$
\STATE $LU=PA,$
\STATE $HY_{0}=L,$
\STATE $Y_{1}=Y_{0}U,$
\STATE $W=XY_{1}^{-1},$
\STATE $QZ=W,$
\STATE $R=ZY_{1}.$
\end{algorithmic}
\end{algorithm}%

However, RCLUPP inherits a key limitation of RCholeskyQR. While sketching reduces the computational cost of the subsequent LUPP and the thin Householder QR computations from $X\in\mathbb{R}^{m\times n}$ to the smaller matrix $A\in\mathbb{R}^{s\times n}$, the rounding errors introduced in this first step ultimately limit its stability. Specifically, the orthogonality analysis involves $\norm{Y^{-1}}_2$, which is bounded by $\kappa_2(L)\kappa_2(U)$ \cite[(4.25)]{RCLUPP}. $\kappa_2(X)$ can be bounded in terms of the same quantity. Therefore, both the theoretical and experimental results indicate that RCLUPP becomes rapidly less stable as $\kappa_2(X)$ increases and eventually suffers numerical breakdown.

To address this limitation, we propose RCLUPPr in \cite[Algorithm~5.1]{RCLUPP}, as shown in Algorithm~\ref{alg:RCLUPPr}. Unlike RCLUPP, RCLUPPr applies LUPP directly to $X\in\mathbb{R}^{m\times n}$ as a preconditioner, thereby avoiding the sketching-induced errors in the orthogonality analysis. Numerical results show that it substantially improves upon RCLUPP, RCholeskyQR and LUC2 in applicability, remaining stable and accurate for ill-conditioned matrices in the double precision. However, rigorous rounding-error bounds for RCLUPPr are still lacking. Also, its LUPP factorization of $X$ introduces moderate overhead relative to RCLUPP and RCholeskyQR. Moreover, its behavior in the mixed-precision arithmetic has not yet been systematically studied.

\begin{algorithm}
\caption{$[Q,R]=\mbox{RCLUPPr}(X)$ with $X \in \mathbb{R}^{m\times n}$}
\label{alg:RCLUPPr}
\begin{algorithmic}[1]
\STATE $LU=PX,$
\STATE $L_{s}=\Omega L,$
\STATE $HY_{0}=L_{s},$
\STATE $Y_{1}=Y_{0}U,$
\STATE $W=XY_{1}^{-1},$
\STATE $QZ=W,$
\STATE $R=ZY_{1}.$
\end{algorithmic}
\end{algorithm}%

\subsection{Innovative points}
This work addresses the limitations of RCLUPPr in several aspects. First, we exhibit a comprehensive rounding error analysis for both orthogonality and residual of the algorithm, establishing its advantages over existing methods. Also, we further extend the analysis and numerical experiments to the mixed-precision arithmetic, showing that RCLUPPr remains stable and accurate while preserving its inherent benefits. To the best of our knowledge, this is also the first study to evaluate the robustness of randomized CholeskyQR-type algorithms. Moreover, we develop several profiling-guided acceleration strategies that reduce the CPU time of RCLUPPr without compromising numerical stability. Finally, we validate the theoretical and practical performance of the proposed method on a diverse collection of real-world matrices, including the SVD-based and the arrowhead matrices, as well as matrices arising from capacitively coupled plasma simulations and the control theory. The scope and main contributions of this work are distinct from our previous work \cite{RCLUPP}. In particular, the numerical analysis of RCLUPPr, the mixed-precision experiments and the acceleration strategies presented here are independent of and do not rely on the theoretical results in \cite{RCLUPP}.

\subsection{Outline and notation}
The rest of the paper is structured into the following sections. Section~\ref{sec:Literature} reviews the theoretical background of rounding error analysis and matrix sketching. In Section~\ref{sec:randomized}, we derive the bounds of both orthogonality and residual for RCLUPPr by rounding error analysis. Section~\ref{sec:comparisons} compares the theoretical results of RCLUPPr with those of the existing algorithms and examines the impact of matrix sketching. Numerical experiments validating the analysis across different precisions are presented in Section~\ref{sec:Numerical}. Finally, Section~\ref{sec:conclusions} summarizes the main findings of this work.

Throughout this paper, $\norm{\cdot}_{F}$ and $\norm{\cdot}_{2}$ are the Frobenius and spectral norms. We take $\sigma_i(X)$ as the $i$-th largest singular value of $X$ with $1 \le i \le n$. Specially, $\norm{X}_{2}=\sigma_{1}(X)$. The condition number $X$, $\kappa_{2}(X)$, is defined by $\kappa_{2}(X)=\frac{\norm{X}_{2}}{\sigma_{n}(X)}$. We let $u$ denote the unit roundoff and $\abs{X}$ represent the entrywise absolute value of a matrix $X$. The operator $\mbox{fl}(\cdot)$ indicates floating-point computation, and $I_{n} \in \mathbb{R}^{n\times n}$ signifies the $n \times n$ identity matrix. For any quantity $A$, its floating-point computed counterpart is designated as $\hat{A}$. This hat notation is adhered to consistently for all intermediate computed objects throughout the text.

\section{Some lemmas}
\label{sec:Literature}
This section presents several fundamental lemmas concerning rounding error analysis and the randomized matrix sketching, which lay the theoretical foundation for our subsequent analysis.

\subsection{Rounding error analysis}
In the beginning, we show some lemmas for rounding error analysis.

\begin{lemma}\cite[Corollary 2.4.4]{MatrixC}
\label{lemma Weyl}
For $A,B \in \mathbb{R}^{m\times n}$, then we have
\begin{equation}
\sigma_{n}(A+B) \ge \sigma_{n}(A)-\norm{B}_{2}. \nonumber
\end{equation}
\end{lemma}

\begin{lemma}\cite[(3.12)]{Higham}
\label{lemma mm}
For $A \in \mathbb{R}^{m\times n}$ and $B \in \mathbb{R}^{n\times p}$, the error in computing the matrix product $AB$ in floating-point arithmetic is bounded by
\begin{equation}
\abs{AB-fl(AB)}\le \gamma_{n}\abs{A}\abs{B}. \nonumber
\end{equation}
Here, $\abs{A}$ is the matrix whose $(i,j)$ element is $\abs{a_{ij}}$ and
\begin{equation}
\gamma_n: = \frac{n{\uu}}{1-n{\uu}} \le 1.02n{\uu}. \nonumber
\end{equation}
\end{lemma}

\begin{lemma}\cite[Theorem 9.3]{Higham}
\label{lemma lu}
Suppose that $L \in \mathbb{R}^{m\times n}$ and $U \in \mathbb{R}^{n\times n}$ are computed LU-factors of $A \in \mathbb{R}^{m\times n}$, then we have
\begin{equation}
\abs{A-LU} \le \gamma_{n}\abs{L}\abs{U}. \nonumber
\end{equation}
\end{lemma}

\begin{lemma}\cite[Theorem 8.5]{Higham}
\label{lemma system}
Let the triangular system $Tx=b$, where $T \in \mathbb{R}^{n\times n}$ is non-singular, be solved by substitution with any ordering. Then, the computed solution $x$ satisfies
\begin{equation}
(T+\Delta T)x=b, \quad \abs{\Delta T}\le \gamma_{n}\abs{T}. \nonumber
\end{equation}
\end{lemma}

\begin{lemma}\cite[Section 18.3]{Higham}
\label{lemma hqr}
Regarding a step of the thin HouseholderQR for the input matrix $X \in \mathbb{R}^{m\times n}$, the computed orthogonal factor $\hat{Q}$ and the upper-triangular factor $\hat{R} \in \mathbb{R}^{n\times n}$ satisfy
\begin{align}
\norm{\hat{Q}^{\top}\hat{Q}-I_{n}}_{F} &\le \mathcal{O}(\uu), \nonumber \\
\norm{\hat{Q}\hat{R}-X}_{F} &\le \gamma_{cmn} \cdot \norm{X}_{F}. \nonumber
\end{align}
Here, $c$ is a positive constant.
\end{lemma}

\subsection{Matrix sketching}
\label{sec:ms}
This subsection briefly introduces matrix sketching and the subspace embedding properties. Readers can refer to \cite{rgs,tool} for further details.

\begin{definition}[$\epsilon$-subspace embedding]\cite[Definition 2.1]{rgs}
\label{definition 1}
We let $\mathcal{K}\subset\mathbb{R}^{n}$ be a subspace. A matrix $\Omega\in\mathbb{R}^{s\times n}$ is an $\epsilon$-subspace embedding for $\mathcal{K}$ when for all $x,y\in\mathcal{K}$, 
\begin{equation}
\left|\langle x,y\rangle-\langle\Omega x,\Omega y\rangle\right|
\leq \epsilon \norm{x}_{2}\norm{y}_{2},
\end{equation}
holds with $0\leq\epsilon<1$.
\end{definition}

Definition~\ref{definition 1} concerns a prescribed subspace. In the real practice, one often uses a random sketching matrix whose distribution provides the embedding property uniformly for any fixed low-dimensional subspace.

\begin{definition}[$(\epsilon,p,n)$ oblivious $l_2$-subspace embedding]\cite[Definition 2.3]{rgs}
\label{definition 2}
A random matrix $\Omega\in\mathbb{R}^{s\times m}$ is an $(\epsilon,p,n)$ oblivious $l_2$-subspace embedding if, for every fixed $n$-dimensional subspace $\mathcal{K}\subset\mathbb{R}^{m}$, $\Omega$ is an $\epsilon$-subspace embedding for $\mathcal{K}$ with probability at least $1-p$.
\end{definition}

With Definition~\ref{definition 2}, we have the following properties of matrix sketching 

\begin{lemma}[Matrix sketching]\cite[Corollary 3.1, Proposition 3.2]{Householder}
\label{lemma 22}
Let $X\in\mathbb{R}^{m\times n}$ and suppose that $\Omega$ is an $\epsilon$-subspace embedding for $\operatorname{range}(X)$. Then, with probability at least $1-p$,
\begin{align}
\frac{\norm{\Omega X}_{2}}{\norm{X}_{2}} &\in [(1-\epsilon)^{\frac{1}{2}}, (1+\epsilon)^{\frac{1}{2}}], \nonumber \\
\frac{\norm{\Omega X}_{F}}{\norm{X}_{F}} &\in [(1-\epsilon)^{\frac{1}{2}}, (1+\epsilon)^{\frac{1}{2}}], \nonumber \\
(1-\epsilon)^{\frac{1}{2}}\,\sigma_n(X)
&\leq \sigma_n(\Omega X)
\leq \norm{\Omega X}_{2}
\leq (1+\epsilon)^{\frac{1}{2}}\,\norm{X}_{2}, \nonumber \\
\frac{\sigma_n(\Omega X)}{(1+\epsilon)^{\frac{1}{2}}}
&\leq \sigma_n(X)
\leq \norm{X}_{2}
\leq \frac{\norm{\Omega X}_{2}}{(1-\epsilon)^{\frac{1}{2}}}. \nonumber
\end{align}
\end{lemma}

Common methods for matrix sketching include the Gaussian sketch, subsampled randomized Hadamard transforms (SRHTs) and the CountSketch. In this work, we consider the Gaussian sketch with the sketch matrix $\Omega=\frac{1}{\sqrt{s}}G$, where $G\in\mathbb{R}^{s\times m}$ has independent entries drawn from $\mathcal{N}(0,1)$. To obtain an $(\epsilon,p,n)$ oblivious $l_2$-subspace embedding, the sketch size can be chosen as \cite{4031351}
\begin{equation}
s=\eta \cdot \frac{log n \cdot log{\frac{1}{p}}}{\epsilon^{2}}. \nonumber
\end{equation}
For $X\in\mathbb{R}^{m\times n}$ with $m\geq n$, one typically takes $n\leq s\leq m$ with $S=\mathcal{O}(n)$.

\section{Theoretical results of RCLUPPr}
\label{sec:randomized}
This section presents the theoretical analysis of RCLUPPr within a mixed-precision framework. We write the starting steps of RCLUPPr in Algorithm~\ref{alg:RCLUPPr} with the error matrices first.
\begin{align}
\hat{L}\hat{U} &= PX+\Delta_{lu}, \label{eq:lu} \\
\hat{L_{s}} &= \Omega \hat{L}+\Delta_{s}, \label{eq:es} \\
\hat{H}\hat{Y}_{0} &= \hat{L_{s}}+\Delta_{h}, \label{eq:eh} \\
\hat{Y}_{1} &=\hat{Y}_{0}\hat{U}+\Delta_{y}, \label{eq:ey} \\
\hat{W}\hat{Y}_{1} &=X+\Delta_{x}, \label{eq:ex} 
\end{align}
Here, $\Delta_{lu}$ in \eqref{eq:elu} is the error matrix of LUPP factorization. $\Delta_{s}$ in \eqref{eq:es} denotes the error matrix of matrix sketching. $\Delta_{h}$ in \eqref{eq:eh} is the error matrix of the thin HouseholderQR. $\Delta_{y}$ in \eqref{eq:ey} is the error generated by matrix multiplications. $\Delta_{x}$ in \eqref{eq:ex} is the error from solving linear systems.

\subsection{Settings for RCLUPPr and the corresponding rounding error analysis}
For the preconditioning steps of RCLUPPr, we propose to increase the accuracy in some cheap lines to get $\hat{W}$, while the remaining lines with the dominant computational costs are computed with the working precision. 

\begin{assumption}
\label{assumption:1}
In Algorithm~\ref{alg:RCLUPPr}, steps $2-4$, which corresponds to \eqref{eq:es}-\eqref{eq:ey}, work with the unit roundoff $\uu_{b}$. The remaining steps, including \eqref{eq:lu}, \eqref{eq:ex} and CholeskyQR2, work with the unit roundoff $\uu$. $\uu_{b}$ is the unit roundoff in the higher accuracy than, or at least equivalent to $\uu$.
\end{assumption}

With Definition~\ref{definition 2}, we have the following assumption. 

\begin{assumption}
\label{assumption:s}
If the sketch matrix $\Omega \in \mathbb{R}^{s\times m}$ is an $(\epsilon,p,n)$ oblivious $l_{2}$-subspace embedding in $\mathbb{R}^{m}$, we let
\begin{equation}
\epsilon < \frac{1-\eta^2}{1+\eta^2},
\quad
\eta
=
\max\left\{
12.8\sqrt{mn\uu+n(n+1)\uu}+0.1,\,
10.44n\sqrt{n}\uu+0.1,\,
0.105
\right\}<1. \nonumber
\end{equation}
\end{assumption}
In Assumption~\ref{assumption:s}, we can also have $\eta \le (\frac{1-\epsilon}{1+\epsilon})^{\frac{1}{2}}$.

For RCLUPPr, we provide the following settings:
\begin{align}
m\sqrt{s}\uu_{b} \le \frac{1}{64} &, \quad 
c_{1}sn\sqrt{n}\uu_{b} \le \frac{1}{64}, \label{eq:a} \\
\kappa_{2}(\hat{L}) \le \frac{(1-\epsilon)^{\frac{1}{2}}}{25.5t+25.5csn\sqrt{n}\uu_{b} \cdot ((1+\epsilon)^{\frac{1}{2}}+t)} &, \quad 
\kappa_{2}(\hat{L})\kappa_{2}(\hat{U}) \le \frac{(1-\epsilon)^{\frac{1}{2}}}{23.6dn^{2}\uu_{c} \cdot ((1+\epsilon)^{\frac{1}{2}}+t)}, \label{eq:b}
\end{align}
with $c_{1}=\max(c,1)$, $t=1.02m\uu_{b} \cdot \sqrt{n}\norm{\Omega}_{F}$, $d=\max(\norm{\Omega}_{2}, 1)$ and $\uu_{c}=\{\uu_{b}, \uu\}$.

Based on the assumptions and settings above, we bound $\norm{\hat{W}}_{2}$ and $\sigma_{n}(\hat{W})$.

\begin{theorem}[Bounding $\norm{\hat{W}}_{2}$ and $\sigma_{n}(\hat{W})$]
\label{theorem:RCLUPPr}
With Assumption~\ref{assumption:1}, Assumption~\ref{assumption:s}, \eqref{eq:a} and \eqref{eq:b}, for $X \in \mathbb{R}^{m\times n}$, then for $\hat{W}$ in \eqref{eq:ex}, we have
\begin{align}
\norm{\hat{W}}_{2} &\le \frac{1.28}{(1-\epsilon)^{\frac{1}{2}}}, \label{eq:w2} \\
\sigma_{n}(\hat{W}) &\le \frac{4}{5(1+\epsilon)^{\frac{1}{2}}}-\frac{0.08}{(1-\epsilon)^{\frac{1}{2}}}, \label{eq:nw}
\end{align}
with probability at least $(1-p)^{2}$.
\end{theorem}

Specially, when RCLUPPr is taken with all the steps in the high precision with $\uu=\uu_{b}$ in Assumption~\ref{assumption:1}, \textit{e.g.}, $\uu_{b}=2^{-53}$ in the double precision, Assumption~\ref{assumption:s} and \eqref{eq:b}, we provide a rounding error analysis of RCLUPPr. Similar analysis can be taken for $\uu$ in the low precision.

\begin{theorem}[Rounding error analysis of RCLUPPr in the high precision]
\label{theorem:RCLUPPrd}
With $\uu=\uu_{b}$ in Assumption~\ref{assumption:1}, Assumption~\ref{assumption:s}, \eqref{eq:a} and \eqref{eq:b}, for $X \in \mathbb{R}^{m\times n}$, then for $[\hat{Q},\hat{R}]=\mbox{RCLUPPr}(X)$, we have
\begin{align}
\norm{\hat{Q}^{\top}\hat{Q}-I_{n}}_{F} &\le 6(mn\uu_{b}+n(n+1)\uu_{b}), \label{eq:ortho} \\
\norm{\hat{Q}\hat{R}-X}_{F} &\le \Delta_{res}, \label{eq:res}
\end{align}
with probability at least $(1-p)^{2}$. Here, we have $\Delta_{res}=[\frac{1.31}{(1-\epsilon)^{\frac{1}{2}}}+\frac{1.37}{(1-\epsilon)^{\frac{1}{2}}} \cdot \sqrt{1+1.03j^{2}}+\frac{1.37}{(1-\epsilon)^{\frac{1}{2}}} \cdot \sqrt{1+6j^{2}} \cdot (1+1.03j^{2})] \cdot kn\sqrt{n}\uu_{b}\norm{X}_{F}+[\frac{1.36}{(1-\epsilon)^{\frac{1}{2}}} \cdot \sqrt{1+1.03j^{2}}+\frac{1.37}{(1-\epsilon)^{\frac{1}{2}}} \cdot \sqrt{1+6j^{2}} \cdot (1+1.03j^{2})] \cdot kn^{2}\uu_{b}\norm{X}_{2}$ with $j=(\frac{\frac{(1-\epsilon)^{\frac{1}{2}}}{(1+\epsilon)^{\frac{1}{2}}}-0.1}{12.8})^{2}$ and $k=\frac{1}{\frac{0.8}{(1+\epsilon)^{\frac{1}{2}}}-\frac{0.1}{(1-\epsilon)^{\frac{1}{2}}}}$.
\end{theorem}

\subsection{Some lemmas of RCLUPPr}
Firstly, we establish several lemmas required for the theoretical analysis of RCLUPPr. The theoretical results in Lemma~\ref{lemma 1}-Lemma~\ref{lemma 7} all hold with a probability of at least $1-p$ subject to the step in \eqref{eq:es}.

\begin{lemma}[Estimation of $\norm{\Delta_{s}}_{F}$]
\label{lemma 1}
For $\Delta_{s}$ in \eqref{eq:es}, with $t=1.02m\uu_{b} \cdot \sqrt{n}\norm{\Omega}_{F}$, we have
\begin{equation}
\norm{\Delta_{s}}_{F} \le t\norm{\hat{L}}_{2}. \label{eq:dsf}
\end{equation}
\end{lemma}
\begin{proof}
For $\Delta_{s}$ in \eqref{eq:es}, with Lemma~\ref{lemma mm} and Assumption~\ref{assumption:1}, $\norm{\Delta_{s}}_{F}$ can be bounded as 
\begin{equation}
\norm{\Delta_{s}}_{F} \le \gamma_{m} \cdot \norm{\Omega}_{F}\norm{\hat{L}}_{F} \le 1.02m\uu_{b} \cdot \norm{\Omega}_{F} \cdot \sqrt{n}\norm{\hat{L}}_{2} = t\norm{\hat{L}}_{2}. \nonumber
\end{equation}
Here, $t=1.02m\uu \cdot \sqrt{n}\norm{\Omega}_{F}$. \eqref{eq:dsf} is proved. Lemma~\ref{lemma 1} holds.
\end{proof}

\begin{lemma}[Estimation of $\norm{\Delta_{h}}_{F}$]
\label{lemma 2}
For $\Delta_{h}$ in \eqref{eq:eh}, we have
\begin{equation}
\norm{\Delta_{h}}_{F} \le 1.02csn\sqrt{n}\uu_{b} \cdot ((1+\epsilon)^{\frac{1}{2}}+t) \cdot \norm{\hat{L}}_{2}. \label{eq:dhf}
\end{equation}
\end{lemma}
\begin{proof}
Before bounding $\norm{\Delta_{h}}_{F}$ in \eqref{eq:eh}, we focus on $\hat{L}_{s}$ first. For $\hat{L}_{s}$ in \eqref{eq:es}, with Lemma~\ref{lemma 22}, \eqref{eq:es} and \eqref{eq:dsf}, we can estimate $\norm{\hat{L}_{s}}_{F}$ as
\begin{equation} \label{eq:ls2}
\norm{\hat{L}_{s}}_{2} \le \norm{\Omega \hat{L}}_{2}+\norm{\Delta_{s}}_{F} \le (1+\epsilon)^{\frac{1}{2}} \cdot \norm{\hat{L}}_{2}+t\norm{\hat{L}}_{2} \le ((1+\epsilon)^{\frac{1}{2}}+t) \cdot \norm{\hat{L}}_{2}. 
\end{equation}
Therefore, for $\Delta_{h}$ in \eqref{eq:eh}, with Lemma~\ref{lemma hqr}, Assumption~\ref{assumption:1} and \eqref{eq:ls2}, we can take the bound of $\norm{\Delta_{h}}_{F}$ as
\begin{equation}
\begin{split}
\norm{\Delta_{h}}_{F} \le \gamma_{csn} \cdot \norm{\hat{L}_{s}}_{F} \le 1.02csn\sqrt{n}\uu_{b} \cdot \norm{\hat{L}_{s}}_{2} \le 1.02csn\sqrt{n}\uu_{b} \cdot ((1+\epsilon)^{\frac{1}{2}}+t) \cdot \norm{\hat{L}}_{2}. \nonumber
\end{split}
\end{equation}
\eqref{eq:dhf} is proved. Lemma~\ref{lemma 2} holds.
\end{proof}

\begin{lemma}[Estimation of $\norm{\hat{Y}_{0}}_{2}$]
\label{lemma 3}
For $\hat{Y}_{0}$ in \eqref{eq:eh}, we have
\begin{equation}
\norm{\hat{Y}_{0}}_{2} \le 1.03((1+\epsilon)^{\frac{1}{2}}+t) \cdot \norm{\hat{L}}_{2}. \label{eq:s2}
\end{equation}
\end{lemma}
\begin{proof}
Before focusing on $\hat{Y}_{0}$ in \eqref{eq:eh}, we turn to $\hat{H}$ in \eqref{eq:eh} first. With Lemma~\ref{lemma hqr}, we can just take $\norm{\hat{H}^{\top}\hat{H}-I_{n}}_{F} \le 0.01$. Therefore, the following inequality holds for $\sigma_{n}(\hat{H})$ and $\norm{\hat{H}}_{2}$ 
\begin{equation}
0.99 \le \sigma_{n}(\hat{H}) \le \norm{\hat{H}}_{2} \le 1.01. \label{eq:eb}
\end{equation}
Therefore, for $\hat{Y}_{0}$ in \eqref{eq:eh}, with Lemma~\ref{lemma 22}, \eqref{eq:a}, \eqref{eq:dhf}, \eqref{eq:ls2} and \eqref{eq:eb}, $\norm{\hat{Y}_{0}}_{2}$ satisfies
\begin{equation}
\begin{split}
\norm{\hat{Y}_{0}}_{2} &\le \norm{\hat{H}\hat{L}_{s}}_{2}+\norm{\Delta_{h}}_{F} \nonumber \\ &\le \norm{\hat{H}}_{2}\norm{\hat{L}_{s}}_{2}+\norm{\Delta_{h}}_{F} \nonumber \\ &\le 1.01((1+\epsilon)^{\frac{1}{2}}+t) \cdot \norm{\hat{L}}_{2}+1.02csn\sqrt{n}\uu_{b} \cdot ((1+\epsilon)^{\frac{1}{2}}+t) \cdot \norm{\hat{L}}_{2} \nonumber \\ &\le (1.01+1.02csn\sqrt{n}\uu_{b}) \cdot ((1+\epsilon)^{\frac{1}{2}}+t) \cdot \norm{\hat{L}}_{2} \nonumber \\ &\le 1.03((1+\epsilon)^{\frac{1}{2}}+t) \cdot \norm{\hat{L}}_{2}. \nonumber
\end{split}
\end{equation}
\eqref{eq:s2} is proved. Lemma~\ref{lemma 3} holds.
\end{proof}

\begin{lemma}[Estimation of $\norm{\Delta_{y}}_{F}$]
\label{lemma 4}
For $\Delta_{y}$ in \eqref{eq:ey}, we have
\begin{equation}
\norm{\Delta_{y}}_{F} \le 1.06n^{2}\uu_{b} \cdot ((1+\epsilon)^{\frac{1}{2}}+t) \cdot \norm{\hat{L}}_{2}\norm{\hat{U}}_{2}. \label{eq:dyf}
\end{equation}
\end{lemma}
\begin{proof}
For $\Delta_{y}$ in \eqref{eq:ey}, with Lemma~\ref{lemma mm}, Assumption~\ref{assumption:1} and \eqref{eq:s2}, we can get
\begin{equation}
\begin{split}
\norm{\Delta_{y}}_{F} &\le 1.02n\uu_{b} \cdot \norm{\hat{Y}_{0}}_{F}\norm{\hat{U}}_{F} \le 1.02n\uu_{b} \cdot \sqrt{n}\norm{\hat{Y}_{0}}_{2} \cdot \sqrt{n}\norm{\hat{U}}_{2} \nonumber \\ &\le 1.02n\uu_{b} \cdot \sqrt{n} \cdot 1.03((1+\epsilon)^{\frac{1}{2}}+t) \cdot \norm{\hat{L}}_{2} \cdot \sqrt{n}\norm{\hat{U}}_{2} \nonumber \\ &\le 1.06n^{2}\uu_{b} \cdot ((1+\epsilon)^{\frac{1}{2}}+t) \cdot \norm{\hat{L}}_{2}\norm{\hat{U}}_{2}. \nonumber
\end{split}
\end{equation}
\eqref{eq:dyf} is proved. Lemma~\ref{lemma 4} holds.
\end{proof}

\begin{lemma}[Estimation of $\norm{\hat{Y}_{1}^{-1}}_{2}$]
\label{lemma 5}
For $\hat{Y}_{1}$ in \eqref{eq:ey}, we have
\begin{equation}
\norm{\hat{Y}_{1}^{-1}}_{2} \le \frac{1.12}{(1-\epsilon)^{\frac{1}{2}}} \cdot \frac{1}{\sigma_{n}(\hat{L})\sigma_{n}(\hat{U})}. \label{eq:y-12}
\end{equation}
\end{lemma}
\begin{proof}
For $\hat{Y}_{1}$ in \eqref{eq:ey}, with \eqref{eq:a}, \eqref{eq:s2} and \eqref{eq:dyf}, we can get
\begin{equation} \label{eq:y2}
\begin{split}
\norm{\hat{Y}_{1}}_{2} &\le \norm{\hat{Y}_{0}}_{2}\norm{\hat{U}}_{2}+\norm{\Delta_{y}}_{F} \\ &\le 1.03 \cdot ((1+\epsilon)^{\frac{1}{2}}+t) \cdot \norm{\hat{L}}_{2} \norm{\hat{U}}_{2} + 1.06n^{2}\uu_{b} \cdot ((1+\epsilon)^{\frac{1}{2}}+t) \cdot \norm{\hat{L}}_{2}\norm{\hat{U}}_{2} \\ &\le 1.05 \cdot ((1+\epsilon)^{\frac{1}{2}}+t) \cdot \norm{\hat{L}}_{2} \norm{\hat{U}}_{2}.
\end{split}
\end{equation}
For $A, B \in \mathbb{R}^{n\times n}$, $\sigma_{n}(AB) \ge \sigma_{n}(A)\sigma_{n}(B)$ holds. With Lemma~\ref{lemma Weyl} and \eqref{eq:ey}, $\sigma_{n}(\hat{Y}_{1})$ satisfies
\begin{equation} \label{eq:ny}
\sigma_{n}(\hat{Y}_{1}) = \sigma_{n}(\hat{Y}_{0}\hat{U}+\Delta_{y}) \ge \sigma_{n}(\hat{Y}_{0})\sigma_{n}(\hat{U})-\norm{\Delta_{y}}_{F}. 
\end{equation}
For $A \in \mathbb{R}^{s\times n}$ and $B \in \mathbb{R}^{n\times n}$, $\sigma_{n}(AB) \le \norm{A}_{2}\sigma_{n}(B)$ holds. Therefore, for $\hat{Y}_{0}$ in \eqref{eq:ny}, with Lemma~\ref{lemma Weyl}, \eqref{eq:eh} and \eqref{eq:eb}, it is clear to see that
\begin{equation} \label{eq:ns1}
\norm{\hat{H}}_{2}\sigma_{n}(\hat{Y}_{0}) \ge \sigma_{n}(\hat{H}\hat{Y}_{0}) =\sigma_{n}(\hat{L_{s}}+\Delta_{h}) \ge \sigma_{n}(\hat{L_{s}})-\norm{\Delta_{h}}_{F}. 
\end{equation}
With \eqref{eq:eb} and \eqref{eq:ns1}, 
\begin{equation} \label{eq:ns}
\sigma_{n}(\hat{Y}_{0}) \ge \frac{\sigma_{n}(\hat{L_{s}})-\norm{\Delta_{h}}_{F}}{\norm{\hat{H}}_{2}} \ge \frac{\sigma_{n}(\hat{L_{s}})-\norm{\Delta_{h}}_{F}}{1.01}. 
\end{equation}
For $\hat{L_{s}}$ in \eqref{eq:es}, with Lemma~\ref{lemma Weyl} and Lemma~\ref{lemma 22}, we can have
\begin{equation} \label{eq:nls}
\begin{split}
\sigma_{n}(\hat{L}_{s}) = \sigma_{n}(\Omega\hat{L}+\Delta_{s}) \ge \sigma_{n}(\Omega\hat{L})-\norm{\Delta_{s}}_{F} \ge (1-\epsilon)^{\frac{1}{2}} \cdot \sigma_{n}(\hat{L})-\norm{\Delta_{s}}_{F}.
\end{split}
\end{equation}
Therefore, we put \eqref{eq:dsf} into \eqref{eq:nls} and we can take the bound of $\sigma_{n}(\hat{L_{s}})$ as
\begin{equation} \label{eq:nls2}
\begin{split}
\sigma_{n}(\hat{L}_{s}) \ge \sigma_{n}(\Omega\hat{L})-\norm{\Delta_{s}}_{F} \ge (1-\epsilon)^{\frac{1}{2}} \cdot \sigma_{n}(\hat{L})-t\norm{\hat{L}}_{2}. 
\end{split}
\end{equation}
With \eqref{eq:a}, \eqref{eq:b}, \eqref{eq:dhf}, \eqref{eq:eb}, \eqref{eq:ns1}, \eqref{eq:ns} and \eqref{eq:nls2}, we bound $\sigma_{n}(\hat{Y}_{0})$ as
\begin{equation} \label{eq:nhs}
\begin{split}
\sigma_{n}(\hat{Y}_{0}) &\ge \frac{\sigma_{n}(\hat{L_{s}})-\norm{\Delta_{h}}_{F}}{\norm{\hat{H}}_{2}} \\ &\ge \frac{(1-\epsilon)^{\frac{1}{2}} \cdot \sigma_{n}(\hat{L})-t\norm{\hat{L}}_{2}-\norm{\Delta_{h}}_{F}}{1.01} \\ &\ge \frac{(1-\epsilon)^{\frac{1}{2}} \cdot \sigma_{n}(\hat{L})-t\norm{\hat{L}}_{2}-1.02csn\sqrt{n}\uu_{b} \cdot ((1+\epsilon)^{\frac{1}{2}}+t) \cdot \norm{\hat{L}}_{2}}{1.01} \\ &\ge \frac{(1-\epsilon)^{\frac{1}{2}} \cdot \sigma_{n}(\hat{L})-[(1+1.02csn\sqrt{n}\uu_{b}) \cdot t+1.02csn\sqrt{n}\uu_{b} \cdot (1+\epsilon)^{\frac{1}{2}}] \cdot \norm{\hat{L}}_{2}}{1.01} \\ &\ge \frac{(1-\epsilon)^{\frac{1}{2}} \cdot \sigma_{n}(\hat{L})-(1.02t+1.02csn\sqrt{n}\uu_{b} \cdot (1+\epsilon)^{\frac{1}{2}}) \cdot \norm{\hat{L}}_{2}}{1.01} \\ &\ge 0.95 \cdot (1-\epsilon)^{\frac{1}{2}} \cdot \sigma_{n}(\hat{L}).
\end{split}
\end{equation}
With \eqref{eq:b}, \eqref{eq:dyf}, \eqref{eq:ny} and \eqref{eq:nhs}, $\sigma_{n}(\hat{Y}_{1})$ has the lower bound of
\begin{equation} \label{eq:ny2}
\begin{split}
\sigma_{n}(\hat{Y}_{1}) &\ge \sigma_{n}(\hat{Y}_{0})\sigma_{n}(\hat{U})-\norm{\Delta_{y}}_{F} \\ &\ge 0.95 \cdot (1-\epsilon)^{\frac{1}{2}} \cdot \sigma_{n}(\hat{L})\sigma_{n}(\hat{U})-1.05n^{2}\uu_{b} \cdot ((1+\epsilon)^{\frac{1}{2}}+t) \cdot \norm{\hat{L}}_{2}\norm{\hat{U}}_{2} \\  &\ge 0.9 \cdot (1-\epsilon)^{\frac{1}{2}} \cdot \sigma_{n}(\hat{L})\sigma_{n}(\hat{U}).
\end{split}
\end{equation}
With \eqref{eq:ny2}, we can get an upper bound of $\norm{\hat{Y}_{1}^{-1}}_{2}$ as 
\begin{equation}
\norm{\hat{Y}_{1}^{-1}}_{2} = \frac{1}{\sigma_{n}(\hat{Y}_{1})} \le \frac{1.12}{(1-\epsilon)^{\frac{1}{2}}} \cdot \frac{1}{\sigma_{n}(\hat{L})\sigma_{n}(\hat{U})}. \nonumber
\end{equation}
\eqref{eq:y-12} is proved. Lemma~\ref{lemma 5} holds.
\end{proof}

\begin{lemma}[Estimation of $\norm{\hat{U}\hat{Y}_{1}^{-1}}_{2}$]
\label{lemma 6}
For $\hat{U}\hat{Y}_{1}^{-1}$ in \eqref{eq:ey}, we have
\begin{equation} \label{eq:uy-12}
\norm{\hat{U}\hat{Y}_{1}^{-1}}_{2} \le \frac{1.12}{(1-\epsilon)^{\frac{1}{2}}} \cdot \frac{1}{\sigma_{n}(\hat{L})}.
\end{equation}
\end{lemma}
\begin{proof}
For $\hat{Y}_{0}$ in \eqref{eq:ey}, with \eqref{eq:nhs}, we can take the bound of $\norm{\hat{Y}_{0}^{-1}}_{2}$ as
\begin{equation} \label{eq:s-12}
\norm{\hat{Y}_{0}^{-1}}_{2} = \frac{1}{\sigma_{n}(\hat{Y}_{0})} \le \frac{1.06}{(1-\epsilon)^{\frac{1}{2}}} \cdot \frac{1}{\sigma_{n}(\hat{L})}.
\end{equation}
For $\hat{U}\hat{Y}_{1}^{-1}$ in \eqref{eq:ey}, we have 
$\hat{Y}_{0}^{-1}=\hat{U}\hat{Y}_{1}^{-1}+\hat{Y}_{0}^{-1}\Delta_{y}\hat{Y}_{1}^{-1}$. Therefore, with \eqref{eq:b}, \eqref{eq:dyf}, \eqref{eq:y-12} and \eqref{eq:s-12}, we bound $\norm{\hat{U}\hat{Y}_{1}^{-1}}_{2}$ as
\begin{equation}
\begin{split}
\norm{\hat{U}\hat{Y}_{1}^{-1}}_{2} &\le \norm{\hat{Y}_{0}^{-1}}_{2}+\norm{\hat{Y}_{0}^{-1}}_{2}\norm{\Delta_{y}}_{F}\norm{\hat{Y}_{1}^{-1}}_{2} \nonumber \\ &\le \frac{1.06}{(1-\epsilon)^{\frac{1}{2}}} \cdot \frac{1}{\sigma_{n}(\hat{L})} \cdot 1.06n^{2}\uu_{b} \cdot ((1+\epsilon)^{\frac{1}{2}}+t) \cdot \norm{\hat{L}}_{2}\norm{\hat{U}}_{2} \cdot \frac{1.12}{(1-\epsilon)^{\frac{1}{2}}} \cdot \frac{1}{\sigma_{n}(\hat{L})\sigma_{n}(\hat{U})} \nonumber \\ &+ \frac{1.06}{(1-\epsilon)^{\frac{1}{2}}} \cdot \frac{1}{\sigma_{n}(\hat{L})} \\ &\le \frac{1.12}{(1-\epsilon)^{\frac{1}{2}}} \cdot \frac{1}{\sigma_{n}(\hat{L})}. \nonumber
\end{split}
\end{equation}
\eqref{eq:uy-12} is proved. Lemma~\ref{lemma 6} holds.
\end{proof}

\begin{lemma}[Estimation of $\norm{\Delta_{x}}_{F}$]
\label{lemma 7}
For $\Delta_{x}$ in \eqref{eq:ex}, we have
\begin{equation}
\norm{\Delta x}_{F} \le 1.08n^{2}\uu \cdot ((1+\epsilon)^{\frac{1}{2}}+t) \cdot \norm{\hat{L}}_{2}\norm{\hat{U}}_{2}\norm{\hat{W}}_{2}. \label{eq:dxf}
\end{equation}
\end{lemma}
\begin{proof}
For $\Delta_{x}$ in \eqref{eq:ex}, we replace $T$ with $\hat{Y}_{1}$, $x$ with $\hat{w}_{i}^{\top}$ and $b$ with $x_{i}^{\top}$ in Lemma~\ref{lemma system}. Here, $\hat{w}_{i}^{\top}$ and $x_{i}^{\top}$ denote the $i$-th column of $\hat{W}$ and $X$. Based on \eqref{eq:ex}, we can have $\abs{\Delta_{xi}^{\top}}=\abs{\hat{w}_{i}^{\top}\Delta \hat{Y}_{1}}$. Here, $\Delta_{xi}^{\top}$ is the $i$-th column of $\Delta_{x}$ in \eqref{eq:ex} with $\Delta \hat{Y}_{1}$ corresponding to $\Delta T$ in Lemma~\ref{lemma system}. With Lemma~\ref{lemma system} and Assumption~\ref{assumption:1}, we can have
\begin{equation} \label{eq:dxf1}
\begin{split}
\norm{\Delta_{x}}_{F} &= \sqrt{\sum_{i=1}^{m}\norm{\abs{\hat{w}_{i}^{\top}\Delta \hat{Y}_{1}}}_{F}^{2}} \\ &\le \gamma_{n} \cdot \norm{\hat{W}}_{F}\norm{\hat{Y}_{1}}_{F} \le 1.02n\uu \cdot \sqrt{n}\norm{\hat{W}}_{2} \cdot \sqrt{n}\norm{\hat{Y}_{1}}_{2} \\ &= 1.02n^{2}\uu \cdot \norm{\hat{W}}_{2}\norm{\hat{Y}_{1}}_{2}.
\end{split}
\end{equation}
Therefore, we put \eqref{eq:y2} into \eqref{eq:dxf1} and we can bound $\norm{\Delta x}_{F}$ as
\begin{equation}
\begin{split}
\norm{\Delta_{x}}_{F} &\le 1.02n^{2}\uu \cdot \norm{\hat{W}}_{2}\norm{\hat{Y}_{1}}_{2} \\ &\le 1.02n^{2}\uu \cdot \norm{\hat{W}}_{2} \cdot 1.05 \cdot ((1+\epsilon)^{\frac{1}{2}}+t) \cdot \norm{\hat{L}}_{2} \norm{\hat{U}}_{2} \\ &\le 1.08n^{2}\uu \cdot ((1+\epsilon)^{\frac{1}{2}}+t) \cdot \norm{\hat{L}}_{2}\norm{\hat{U}}_{2}\norm{\hat{W}}_{2}. \nonumber
\end{split}
\end{equation}
\eqref{eq:dxf} is proved. Lemma~\ref{lemma 7} holds
\end{proof}

\subsection{Proof of Theorem~\ref{theorem:RCLUPPr}}
With the previous lemmas, we prove Theorem~\ref{theorem:RCLUPPr}.

\begin{proof}
In the following, we prove \eqref{eq:w2} and \eqref{eq:nw} for Theorem~\ref{theorem:RCLUPPr}.

We start with bounding $\norm{X\hat{Y}_{1}^{-1}}_{2}$ and $\norm{\hat{W}-X\hat{Y}_{1}^{-1}}_{2}$ in \eqref{eq:ex} first. For $X\hat{Y}_{1}^{-1}$ in \eqref{eq:ex}, with \eqref{eq:lu}-\eqref{eq:ex}, we can have
\begin{equation} \label{eq:zheng}
\begin{split}
\Omega PX\hat{Y}_{1}^{-1} &= \Omega \hat{L}\hat{U}\hat{Y}_{1}^{-1}-\Omega \Delta_{lu}\hat{Y}_{1}^{-1} \\ &= (\hat{L}_{s}-\Delta_{s})\hat{U}\hat{Y}_{1}^{-1}-\Omega \Delta_{lu}\hat{Y}_{1}^{-1} \\ &= (\hat{H}\hat{Y}_{0}-\Delta_{h}-\Delta_{s})\hat{U}\hat{Y}_{1}^{-1}-\Omega \Delta_{lu}\hat{Y}_{1}^{-1} \\ &= \hat{H}\hat{Y}_{0}\hat{U}\hat{Y}_{1}^{-1}-\Delta_{h}\hat{U}\hat{Y}_{1}^{-1}-\Delta_{s}\hat{U}\hat{Y}_{1}^{-1}-\Omega \Delta_{lu}\hat{Y}_{1}^{-1} \\ &= \hat{H}(\hat{Y}_{1}-\Delta_{y})\hat{Y}_{1}^{-1}-\Delta_{h}\hat{U}\hat{Y}_{1}^{-1}-\Delta_{s}\hat{U}\hat{Y}_{1}^{-1}-\Omega \Delta_{lu}\hat{Y}_{1}^{-1} \\ &= \hat{H}-\hat{H}\Delta_{y}\hat{Y}_{1}^{-1}-\Delta_{h}\hat{U}\hat{Y}_{1}^{-1}-\Delta_{s}\hat{U}\hat{Y}_{1}^{-1}-\Omega \Delta_{lu}\hat{Y}_{1}^{-1}. 
\end{split}
\end{equation}
With \eqref{eq:zheng}, it is clear that
\begin{equation} \label{eq:bound}
\begin{split}
\norm{\Omega PX\hat{Y}_{1}^{-1}-\hat{H}}_{2} &\le \norm{\hat{H}}_{2}\norm{\Delta_{y}}_{F}\norm{\hat{Y}_{1}^{-1}}_{2}+\norm{\Delta_{h}}_{F}\norm{\hat{U}\hat{Y}_{1}^{-1}}_{2} \\ &+ \norm{\Delta_{s}}_{F}\norm{\hat{U}\hat{Y}_{1}^{-1}}_{2}+\norm{\Omega}_{2}\norm{\Delta_{lu}}_{F}\norm{\hat{Y}_{1}^{-1}}_{2}. 
\end{split}
\end{equation}
For $\Delta_{lu}$ in \eqref{eq:lu}, with Lemma~\ref{lemma lu} and Assumption~\ref{assumption:1}, we can take the bound of $\norm{\Delta_{lu}}_{F}$ as
\begin{equation} \label{eq:elu}
\begin{split}
\norm{\Delta_{lu}}_{F} \le \gamma_{n} \cdot \norm{\hat{L}}_{F}\norm{\hat{U}}_{F} \le 1.02n\uu \cdot \sqrt{n}\norm{\hat{L}}_{2} \cdot \sqrt{n}\norm{\hat{U}}_{2} \le 1.02n^{2}\uu \cdot \norm{\hat{L}}_{2}\norm{\hat{U}}_{2}.
\end{split}
\end{equation}
Therefore, with \eqref{eq:b}, we put \eqref{eq:dsf}, \eqref{eq:dhf}, \eqref{eq:eb}-\eqref{eq:y-12}, \eqref{eq:uy-12} and \eqref{eq:elu} into \eqref{eq:bound}, we can get
\begin{equation} \label{eq:pxy-1}
\begin{split}
\norm{\Omega PX\hat{Y}_{1}^{-1}-\hat{H}}_{2} &\le 1.01 \cdot 1.06n^{2}\uu_{b} \cdot ((1+\epsilon)^{\frac{1}{2}}+t) \cdot \norm{\hat{L}}_{2}\norm{\hat{U}}_{2} \cdot \frac{1.12}{(1-\epsilon)^{\frac{1}{2}}} \cdot \frac{1}{\sigma_{n}(\hat{L})\sigma_{n}(\hat{U})} \\ &+1.02csn\sqrt{n}\uu_{b} \cdot ((1+\epsilon)^{\frac{1}{2}}+t) \cdot \norm{\hat{L}}_{2} \cdot \frac{1.12}{(1-\epsilon)^{\frac{1}{2}}} \cdot \frac{1}{\sigma_{n}(\hat{L})} \\ &+t\norm{\hat{L}}_{2} \cdot \frac{1.12}{(1-\epsilon)^{\frac{1}{2}}} \cdot \frac{1}{\sigma_{n}(\hat{L})}  \\ &+1.02n^{2}\uu \cdot \norm{\hat{L}}_{2}\norm{\hat{U}}_{2} \cdot \norm{\Omega}_{2} \cdot \frac{1.12}{(1-\epsilon)^{\frac{1}{2}}} \cdot \frac{1}{\sigma_{n}(\hat{L})\sigma_{n}(\hat{U})} \\ &\le \frac{1.01 \cdot 1.06 \cdot 1.12+1.02 \cdot 1.12}{23.6}+\frac{1.02 \cdot 1.12}{25.5}+\frac{1.12}{25.5} \\ &\le 0.19.
\end{split}
\end{equation}
Therefore, with Lemma~\ref{lemma Weyl}, \eqref{eq:eb} and \eqref{eq:pxy-1}, 
\begin{equation} \label{eq:eq1}
0.8 \le \sigma_{n}(\hat{H})-0.19 \le \sigma_{n}(\Omega PX\hat{Y}_{1}^{-1}) \le \norm{\Omega PX\hat{Y}_{1}^{-1}}_{2} \le \norm{\hat{H}}_{2}+0.19 \le 1.2.
\end{equation}
For $X\hat{Y}_{1}^{-1}$ in \eqref{eq:ex}, when $P \in \mathbb{R}^{m\times m}$ is a permutation matrix and $\Omega$ embeds $PX\hat{Y}_{1}^{-1}$, with Lemma~\ref{lemma 22} and \eqref{eq:eq1}, 
\begin{equation} \label{eq:eq3}
\frac{4}{5(1+\epsilon)^{\frac{1}{2}}} \le \sigma_{n}(PX\hat{Y}_{1}^{-1})=\sigma_{n}(X\hat{Y}_{1}^{-1}) \le \norm{X\hat{Y}_{1}^{-1}}_{2}=\norm{PX\hat{Y}_{1}^{-1}}_{2} \le \frac{6}{5(1-\epsilon)^{\frac{1}{2}}}. 
\end{equation}
\eqref{eq:pxy-1} and \eqref{eq:eq1} hold with probability at leat $(1-p)^{2}$. Starting from \eqref{eq:eq3}, all the theoretical results hold with probability at least $(1-p)^{2}$.
For $\hat{W}-X\hat{Y}_{1}^{-1}$ in \eqref{eq:ex}, with \eqref{eq:b}, \eqref{eq:y-12} and \eqref{eq:dxf}, $\norm{\hat{W}-X\hat{Y}_{1}^{-1}}_{2}$ satisfies
\begin{equation} \label{eq:w-xy-1}
\begin{split}
\norm{\hat{W}-X\hat{Y}_{1}^{-1}}_{2} &= \norm{\Delta_{x}\hat{Y}_{1}^{-1}} \le \norm{\Delta_{x}}_{F}\norm{\hat{Y}_{1}^{-1}}_{2} \\ &\le 1.08n^{2}\uu \cdot ((1+\epsilon)^{\frac{1}{2}}+t) \cdot \norm{\hat{L}}_{2}\norm{\hat{U}}_{2}\norm{\hat{W}}_{2} \cdot \frac{1.12}{(1-\epsilon)^{\frac{1}{2}}} \cdot \frac{1}{\sigma_{n}(\hat{L})\sigma_{n}(\hat{U})} \\ &\le 0.06\norm{\hat{W}}_{2}.
\end{split}
\end{equation}

Therefore, for $\hat{W}$ in \eqref{eq:ex}, with Lemma~\ref{lemma Weyl} and \eqref{eq:ex}, we can have
\begin{equation}
\sigma_{n}(X\hat{Y}_{1}^{-1})-\norm{\hat{W}-X\hat{Y}_{1}^{-1}}_{2} \le \sigma_{n}(\hat{W}) \le \norm{\hat{W}}_{2} \le \norm{X\hat{Y}_{1}^{-1}}_{2}+\norm{\hat{W}-X\hat{Y}_{1}^{-1}}_{2}. \label{eq:sigma}
\end{equation}
With \eqref{eq:eq3}-\eqref{eq:sigma}, we can get
\begin{equation} \label{eq:w21}
\norm{\hat{W}}_{2} \le \frac{6}{5(1-\epsilon)^{\frac{1}{2}}}+0.06\norm{\hat{W}}_{2}.
\end{equation}
With \eqref{eq:w21}, $\norm{\hat{W}}_{2}$ satisfies
\begin{equation} 
\norm{\hat{W}}_{2} \le \frac{1.28}{(1-\epsilon)^{\frac{1}{2}}}. \nonumber
\end{equation}
\eqref{eq:w2} is proved. With \eqref{eq:eq3}-\eqref{eq:sigma} and \eqref{eq:w2}, we can get a lower bound of $\sigma_{n}(\hat{W})$ as
\begin{equation} 
\begin{split}
\sigma_{n}(\hat{W}) &\ge \sigma_{n}(X\hat{Y}_{1}^{-1})-\norm{\hat{W}-X\hat{Y}_{1}^{-1}}_{2} \\ &\ge \sigma_{n}(X\hat{Y}_{1}^{-1})-0.06 \cdot \norm{\hat{W}}_{2} \\ &\ge \frac{4}{5(1+\epsilon)^{\frac{1}{2}}}-0.06 \cdot \frac{1.28}{(1-\epsilon)^{\frac{1}{2}}} \\ &\ge \frac{4}{5(1+\epsilon)^{\frac{1}{2}}}-\frac{0.08}{(1-\epsilon)^{\frac{1}{2}}}. \nonumber
\end{split}
\end{equation}
\eqref{eq:nw} is proved. Therefore, Theorem~\ref{theorem:RCLUPPr} holds.
\end{proof}

\subsection{Proof of Theorem~\ref{theorem:RCLUPPrd}}
With the previous theoretical results, we prove Theorem~\ref{theorem:RCLUPPrd}. 

\begin{proof}
We divide the proof into two steps to bound the orthogonality $\norm{\hat{Q}^{\top}\hat{Q}-I_{n}}_{F}$ and the residual $\norm{\hat{Q}\hat{R}-X}_{F}$.

\subsubsection{Orthogonality}
For $\norm{\hat{Q}^{\top}\hat{Q}-I_{n}}_{F}$, we focus on $\kappa_{2}(\hat{W})$ first. With \eqref{eq:w2} and \eqref{eq:nw}, $\kappa_{2}(\hat{W})$ has the upper bound of
\begin{equation} \label{eq:k2}
\kappa_{2}(\hat{W}) \le \frac{\norm{\hat{W}}_{2}}{\sigma_{n}(\hat{W})} \le \frac{\frac{1.28}{(1-\epsilon)^{\frac{1}{2}}}}{\frac{4}{5(1+\epsilon)^{\frac{1}{2}}}-\frac{0.08}{(1-\epsilon)^{\frac{1}{2}}}} \le \frac{1.28}{0.8 \cdot (\frac{1-\epsilon}{1+\epsilon})^{\frac{1}{2}}-0.08}.
\end{equation}

For $\norm{\hat{Q}^{\top}\hat{Q}-I_{n}}_{F}$, with Assumption~\ref{assumption:s} and \eqref{eq:k2}, when $\uu=\uu_{b}$, it is clear that
\begin{equation} \label{eq:set}
\frac{1.28}{0.8 \cdot (\frac{1-\epsilon}{1+\epsilon})^{\frac{1}{2}}-0.08} \le \frac{1}{8\sqrt{mn\uu_{b}+n(n+1)\uu_{b}}},
\end{equation}
Therefore, with \cite[Theorem 3.3]{error} and \eqref{eq:set}, $\norm{\hat{Q}^{\top}\hat{Q}-I_{n}}_{F}$ satisfies
\begin{equation}
\norm{\hat{Q}^{\top}\hat{Q}-I_{n}}_{F} \le 6(mn\uu_{b}+n(n+1)\uu_{b}). \nonumber
\end{equation}
\eqref{eq:ortho} is proved.

\subsubsection{Residual}
In the following, we turn to the residual $\norm{\hat{Q}\hat{R}-X}_{F}$. CholeskyQR2 following the preconditioning steps can be explicitly expressed by incorporating the error matrices below.
\begin{align}
\hat{G}_{1} &= \hat{W}^{\top}\hat{W}+\Delta_{1}, \label{eq:C1} \\
\hat{Y}_{2}^{\top}\hat{Y}_{2} &= \hat{G}_{1}+\Delta_{2}, \label{eq:C2} \\
\hat{V}\hat{Y}_{2} &= \hat{W}+\Delta_{3}, \label{eq:C3} \\
\hat{Y}_{3} &= \hat{Y}_{2}\hat{Y}_{1}+\Delta_{4}, \label{eq:C4} \\
\hat{G}_{2} &= \hat{V}^{\top}\hat{V}+\Delta_{5}, \label{eq:C5} \\
\hat{Y}_{4}^{\top}\hat{Y}_{4} &= \hat{G}_{2}+\Delta_{6}, \label{eq:C6} \\
\hat{Q}\hat{Y}_{4} &= \hat{V}+\Delta_{7}, \label{eq:C7} \\
\hat{R} &= \hat{Y}_{4}\hat{Y}_{3}+\Delta_{8}. \label{eq:C8}
\end{align}
With \eqref{eq:C3}, \eqref{eq:C4}, \eqref{eq:C7} and \eqref{eq:C8}, we can have
\begin{equation} \label{eq:qr-x}
\begin{split}
\hat{Q}\hat{R} &= (\hat{V}+\Delta_{7})\hat{Y}_{4}^{-1}(\hat{Y}_{4}\hat{Y}_{3}+\Delta_{8}) \\ &= \hat{V}\hat{Y}_{3}+\Delta_{7}\hat{Y}_{3}+\hat{Q}\Delta_{8} \\ &= (\hat{W}+\Delta_{3})\hat{Y}_{2}^{-1}(\hat{Y}_{2}\hat{Y}_{1}+\Delta_{4})+\Delta_{7}\hat{Y}_{3}+\hat{Q}\Delta_{8} \\ &= \hat{W}\hat{Y}_{1}+\Delta_{3}\hat{Y}_{1}+\hat{V}\Delta_{4}+\Delta_{7}\hat{Y}_{3}+\hat{Q}\Delta_{8}. 
\end{split}
\end{equation}
With \eqref{eq:qr-x}, for the residual, 
\begin{equation} \label{eq:qr-xf}
\begin{split}
\norm{\hat{Q}\hat{R}-X}_{F} &\le \norm{\hat{W}\hat{Y}_{1}-X}_{F}+\norm{\Delta_{3}}_{F}\norm{\hat{Y}_{1}}_{2}+\norm{\hat{V}}_{2}\norm{\Delta_{4}}_{F} \\ &+\norm{\Delta_{7}}_{F}\norm{\hat{Y}_{3}}_{2}+\norm{\hat{Q}}_{2}\norm{\Delta_{8}}_{F}.
\end{split}
\end{equation}
Therefore, we need to estimate $\norm{\hat{W}\hat{Y}_{1}-X}_{F}$, $\norm{\Delta_{3}}_{F}$, $\norm{\hat{Y}_{1}}_{2}$, $\norm{\hat{V}}_{2}$, $\norm{\Delta_{4}}_{F}$, $\norm{\Delta_{7}}_{F}$, $\norm{\hat{Y}_{3}}_{2}$, $\norm{\hat{Q}}_{2}$ and $\norm{\Delta_{8}}_{F}$ to bound $\norm{\hat{Q}\hat{R}-X}_{F}$ according to \eqref{eq:qr-xf}.

According to Assumption~\ref{assumption:s}, we define $j=(\frac{\frac{(1-\epsilon)^{\frac{1}{2}}}{(1+\epsilon)^{\frac{1}{2}}}-0.1}{12.8})^{2}$. With $m>n$ and $\frac{(1-\epsilon)^{\frac{1}{2}}}{(1+\epsilon)^{\frac{1}{2}}}<1$ based on $0 \le \epsilon<1$, when $\uu=\uu_{b}$, we can find that
\begin{equation} \label{eq:rest1}
2n(n+1)\uu_{b} \le mn\uu_{b}+n(n+1)\uu_{b} \le j^{2}=(\frac{\frac{(1-\epsilon)^{\frac{1}{2}}}{(1+\epsilon)^{\frac{1}{2}}}-0.1}{12.8})^{2} \le 0.005.
\end{equation}
With \eqref{eq:rest1}, we can get
\begin{equation} \label{eq:rest}
n\sqrt{n}\uu_{b} \le n^{2}\uu_{b} \le n(n+1)\uu_{b} \le 0.003.
\end{equation}
Here, \eqref{eq:rest} is a stricter condition than \eqref{eq:a}. For $\Delta_{x}=\hat{W}\hat{Y}_{1}-X$ in \eqref{eq:ex}, with Lemma~\ref{lemma Weyl}, \eqref{eq:dxf1}, the properties of the matrix norms and the singular values, we can have
\begin{equation} \label{eq:xf}
\begin{split}
\norm{\hat{W}\hat{Y}_{1}-\Delta_{x}}_{F} \ge \norm{\hat{W}\hat{Y}_{1}}_{F}-\norm{\Delta_{x}}_{F} \ge \sigma_{n}(\hat{W})\norm{\hat{Y}_{1}}_{F}-1.02n\sqrt{n}\uu_{b} \cdot \norm{\hat{W}}_{2}\norm{\hat{Y}_{1}}_{F}. 
\end{split}
\end{equation}
Assumption~\ref{assumption:s} and \eqref{eq:k2} guarantees $\sigma_{n}(\hat{W})\norm{\hat{Y}_{1}}_{F}-1.02n\sqrt{n}\uu_{b} \cdot \norm{\hat{W}}_{2}\norm{\hat{Y}_{1}}_{F}>0$. Therefore, with \eqref{eq:w2}, \eqref{eq:nw}, \eqref{eq:rest} and \eqref{eq:xf}, $\norm{\hat{Y}_{1}}_{F}$ can be bounded as
\begin{equation} \label{eq:yf2}
\begin{split}
\norm{\hat{Y}_{1}}_{F} &\le \frac{\norm{X}_{F}}
{\sigma_{n}(\hat{W})-1.02n\sqrt{n}\uu_{b}\norm{\hat{W}}_{2}} \\
&\le \frac{\norm{X}_{F}}
{\frac{4}{5(1+\epsilon)^{\frac{1}{2}}}-\frac{0.08}{(1-\epsilon)^{\frac{1}{2}}}
-1.02n\sqrt{n}\uu_{b}\cdot\frac{1.28}{(1-\epsilon)^{\frac{1}{2}}}} \\
&\le \frac{\norm{X}_{F}}
{\frac{0.8}{(1+\epsilon)^{\frac{1}{2}}}-\frac{0.084}{(1-\epsilon)^{\frac{1}{2}}}}=
{k\norm{X}_{F}}.
\end{split}
\end{equation}
Here, we define $k=\frac{1}{\frac{0.8}{(1+\epsilon)^{\frac{1}{2}}}-\frac{0.084}{(1-\epsilon)^{\frac{1}{2}}}}$. It is clear that $k > 0$ with Assumption~\ref{assumption:s}. Therefore, when $\uu=\uu_{b}$, with \eqref{eq:w2}, \eqref{eq:dxf1} and \eqref{eq:yf2}, we can take the bound of $\norm{\Delta_{x}}_{F}=\norm{\hat{W}\hat{Y}_{1}-X}_{F}$ as
\begin{equation} \label{eq:wy-xf}
\begin{split}
\norm{\Delta_{xy}}_{F} \le 1.02n\sqrt{n}\uu_{b} \cdot \norm{\hat{W}}_{2}\norm{\hat{Y}_{1}}_{F} \le 1.02n\sqrt{n}\uu_{b} \cdot \frac{1.28k\norm{X}_{F}}{(1-\epsilon)^{\frac{1}{2}}} \le \frac{1.31kn\sqrt{n}\uu_{b} \cdot \norm{X}_{F}}{(1-\epsilon)^{\frac{1}{2}}}.
\end{split}
\end{equation}

For $\hat{Y}_{1}$ in \eqref{eq:ey}, in a similar way to get \eqref{eq:yf2}, $\norm{\hat{Y}_{1}}_{2}$ has the upper bound of
\begin{equation} \label{eq:y2a}
\norm{\hat{Y}_{1}}_{2} \le k\norm{X}_{2}.
\end{equation}
For $\Delta_{3}$ and $\Delta_{4}$ in \eqref{eq:C3} and \eqref{eq:C4}, we need to focus on $\hat{Y}_{2}$ and $\hat{V}$ in \eqref{eq:C2} and \eqref{eq:C3} first with the analysis of $\Delta_{1}$ and $\Delta_{2}$ in \eqref{eq:C1} and \eqref{eq:C2}. According to \cite[(3.7), (3.13)]{error}, with Lemma~\ref{lemma mm}, Assumption~\ref{assumption:1}, \eqref{eq:rest1} and \eqref{eq:rest}, when $\uu=\uu_{b}$, $\norm{\Delta_{1}}_{F}$ and $\norm{\Delta_{2}}_{F}$ satisfies
\begin{align} 
\norm{\Delta_{1}}_{F} &\le \gamma_{m} \cdot n\norm{\hat{W}}_{2}^{2} \le 1.02mn\uu_{b} \cdot \norm{X}_{2}^{2}, \label{eq:d1f} \\
\begin{split} 
\norm{\Delta_{2}}_{F} &\le \frac{\gamma_{n+1} \cdot n(1+\gamma_{m}n)}{1-\gamma_{n+1}n} \cdot \norm{X}_{2}^{2} \\ &\le 1.02n(n+1)\uu_{b} \cdot \frac{1+1.02mn\uu_{b}}{1-1.02n(n+1)\uu_{b}} \cdot \norm{X}_{2}^{2} \\ &\le 1.02n(n+1)\uu_{b} \cdot \frac{1+0.0051}{1-0.0031} \cdot \norm{X}_{2}^{2} \\ &\le 1.03n(n+1)\uu_{b} \cdot \norm{X}_{2}^{2}.
\end{split}
\label{eq:d2f}
\end{align}
For $\hat{Y}_{2}$ in \eqref{eq:C2}, with \eqref{eq:C1}, \eqref{eq:C2}, \eqref{eq:rest1}, \eqref{eq:d1f} and \eqref{eq:d2f},  we can have
\begin{equation} \label{eq:d22}
\begin{split}
\norm{\hat{Y}_{2}}_{2}^{2} &\le \norm{\hat{W}}_{2}^{2}+\norm{\Delta_{1}}_{F}+\norm{\Delta_{2}}_{F} \\ &\le \norm{\hat{W}}_{2}^{2}+1.03(mn\uu_{b}+n(n+1)\uu_{b}) \cdot \norm{\hat{W}}_{2}^{2} \\ &\le (1+1.03j^{2}) \cdot \norm{\hat{W}}_{2}^{2}.
\end{split}
\end{equation}
Therefore, with \eqref{eq:w2}, we can take the upper bound of $\norm{\hat{Y}_{2}}_{2}$ as
\begin{equation} \label{eq:d2}
\norm{\hat{Y}_{2}}_{2} \le \sqrt{1+1.03j^{2}} \cdot \norm{\hat{W}}_{2} \le \sqrt{1+1.03j^{2}} \cdot \frac{1.28}{(1-\epsilon)^{\frac{1}{2}}}. 
\end{equation}
For $\hat{V}$ in \eqref{eq:C3}, with \cite[(3.26)]{error} and \eqref{eq:rest1}, we have
\begin{equation} \label{eq:v2}
\norm{\hat{V}}_{2} \le \frac{\sqrt{69}}{8}.
\end{equation}
Therefore, following the step to get \eqref{eq:dxf1}, with Lemma~\ref{lemma mm}, Lemma~\ref{lemma system}, Assumption~\ref{assumption:1}, \eqref{eq:yf2}, \eqref{eq:d2} and \eqref{eq:v2}, when $\uu=\uu_{b}$, we can bound $\norm{\Delta_{3}}_{F}$ and $\norm{\Delta_{4}}_{F}$ as
\begin{align} 
\begin{split}
\norm{\Delta_{3}}_{F} &\le \gamma_{n} \cdot \norm{\hat{V}}_{F}\norm{\hat{Y}_{2}}_{F} \le 1.02n\uu_{b} \cdot \sqrt{n}\norm{\hat{V}}_{2} \cdot \sqrt{n}\norm{\hat{Y}_{2}}_{2} \\ &\le 1.02n\uu_{b} \cdot \frac{\sqrt{69n}}{8} \cdot \sqrt{1+1.03j^{2}} \cdot \frac{1.28\sqrt{n}}{(1-\epsilon)^{\frac{1}{2}}} \\ &\le \frac{1.36}{(1-\epsilon)^{\frac{1}{2}}} \cdot \sqrt{1+1.03j^{2}} \cdot n^{2}\uu_{b}, \\  
\end{split}
\label{eq:d3f} \\ 
\begin{split}
\norm{\Delta_{4}}_{F} &\le \gamma_{n} \cdot \norm{\hat{Y}_{2}}_{F}\norm{\hat{Y}_{1}}_{F} \le 1.02n\uu _{b} \cdot \sqrt{n}\norm{\hat{Y}_{2}}_{2} \cdot \norm{\hat{Y}_{1}}_{F} \\ &\le 1.02n\uu_{b} \cdot \sqrt{1+1.03j^{2}} \cdot \frac{1.28\sqrt{n}}{(1-\epsilon)^{\frac{1}{2}}} \cdot k\norm{X}_{F} \\ &\le \frac{1.31}{(1-\epsilon)^{\frac{1}{2}}} \cdot \sqrt{1+1.03j^{2}} \cdot kn\sqrt{n}\uu_{b} \cdot \norm{X}_{F}.
\end{split}
\label{eq:d4f} 
\end{align}

For $\Delta_{7}$ and $\Delta_{8}$ in \eqref{eq:C7} and \eqref{eq:C8}, we need to focus on $\hat{Y}_{3}$, $\hat{Y}_{4}$ and $\hat{Q}$ in \eqref{eq:C4}, \eqref{eq:C6} and \eqref{eq:C7} first. For $\hat{Y}_{3}$ in \eqref{eq:C4}, with \eqref{eq:rest}, \eqref{eq:yf2}, \eqref{eq:y2a}, \eqref{eq:d2} and \eqref{eq:d4f}, since $\frac{\norm{X}_{F}}{\norm{X}_{2}} \le \sqrt{n}$, $\norm{\hat{Y}_{3}}_{l}$ satisfies
\begin{align} 
\begin{split}
\norm{\hat{Y}_{3}}_{l} &\le \norm{\hat{Y}_{2}}_{2}\norm{\hat{Y}_{1}}_{l}+\norm{\Delta_{4}}_{F} \\ &\le \sqrt{1+1.03j^{2}} \cdot \frac{1.28}{(1-\epsilon)^{\frac{1}{2}}} \cdot k\norm{X}_{l} +\frac{1.31}{(1-\epsilon)^{\frac{1}{2}}} \cdot \sqrt{1+1.03j^{2}} \cdot kn\sqrt{n}\uu_{b} \cdot \norm{X}_{F} \\ &\le \sqrt{1+1.03j^{2}} \cdot \frac{1.29}{(1-\epsilon)^{\frac{1}{2}}} \cdot k\norm{X}_{l},
\end{split}
\label{eq:n2},
\end{align}
with $l=\{2,F\}$. For $\hat{Y}_{4}$ in \eqref{eq:C6}, similar to the steps to get \eqref{eq:d2} and with \eqref{eq:v2}, $\norm{\hat{Y}_{4}}_{2}$ satisfies
\begin{equation} \label{eq:j22}
\norm{\hat{Y}_{4}}_{2} \le \sqrt{1+1.03j^{2}} \cdot \norm{\hat{V}}_{2} \le \sqrt{1+1.03j^{2}} \cdot \frac{\sqrt{69}}{8}.
\end{equation}
For $\hat{Q}$ in \eqref{eq:C7}, with \eqref{eq:ortho} and \eqref{eq:rest1}, 
\begin{equation} \label{eq:q2}
\norm{\hat{Q}}_{2} \le \sqrt{\norm{I_{n}}_{2}+6(mn\uu_{b}+n(n+1)\uu_{b})} \le \sqrt{1+6j^{2}}.
\end{equation}
Therefore, following the steps to get \eqref{eq:dxf1}, \eqref{eq:d3f} and \eqref{eq:d4f}, with Lemma~\ref{lemma mm}, Lemma~\ref{lemma system}, Assumption~\ref{assumption:1} and \eqref{eq:n2}-\eqref{eq:q2}, when $\uu=\uu_{b}$, $\norm{\Delta_{7}}_{F}$ and $\norm{\Delta_{8}}_{F}$ have the upper bounds of
\begin{align} 
\begin{split} 
\norm{\Delta_{7}}_{F} &\le \gamma_{n} \cdot \norm{\hat{Q}}_{F}\norm{\hat{Y}_{4}}_{F} \le 1.02n\uu_{b} \cdot \sqrt{n}\norm{\hat{Q}}_{2} \cdot \sqrt{n}\norm{\hat{Y}_{4}}_{2} \\ &\le 1.02n\uu_{b} \cdot \sqrt{1+6j^{2}} \cdot \sqrt{n} \cdot \sqrt{1+1.03j^{2}} \cdot \frac{\sqrt{69}}{8} \cdot \sqrt{n} \\ &\le 1.06n^{2}\uu_{b} \cdot \sqrt{1+6j^{2}} \cdot \sqrt{1+1.03j^{2}}, 
\end{split}
\label{eq:d7f} \\
\begin{split} 
\norm{\Delta_{8}}_{F} &\le \gamma_{n} \cdot \norm{\hat{Y}_{4}}_{F}\norm{\hat{Y}_{3}}_{F} \le 1.02n\uu_{b} \cdot \sqrt{n}\norm{\hat{Y}_{4}}_{2} \cdot \norm{\hat{Y}_{3}}_{F} \\ &\le 1.02n\uu_{b} \cdot \sqrt{1+1.03j^{2}} \cdot \frac{\sqrt{69}}{8} \cdot \sqrt{n} \cdot \sqrt{1+1.03j^{2}} \cdot \frac{1.29}{(1-\epsilon)^{\frac{1}{2}}} \cdot k\norm{X}_{F} \\ &\le \frac{1.37}{(1-\epsilon)^{\frac{1}{2}}} \cdot (1+1.03j^{2}) \cdot kn\sqrt{n}\uu_{b} \cdot \norm{X}_{F}. 
\end{split}
\label{eq:d8f}
\end{align}
Therefore, \eqref{eq:res} is received when we put \eqref{eq:wy-xf}, \eqref{eq:y2a}, \eqref{eq:v2}-\eqref{eq:n2} and \eqref{eq:q2}-\eqref{eq:d8f} into \eqref{eq:qr-xf}. In all, Theorem~\ref{theorem:RCLUPPrd} holds.
\end{proof}

\section{Discussions}
\label{sec:comparisons}
In this section, some discussions of RCLUPPr are provided. We focus on the applicability of RCLUPPr first in the uniform precision. We take $\uu_{b}=\uu$ in Assumption~\ref{assumption:1}. Therefore, $\uu_{c}=\uu$ holds in \eqref{eq:b}. Since $\norm{\Delta_{lu}}_{F}$ in \eqref{eq:lu} is very small in most real cases compared with the theoretical bound in \eqref{eq:elu}, with \eqref{eq:b} and \eqref{eq:elu}, we can have
\begin{equation}
\kappa_{2}(X) = \kappa_{2}(PX) \le a\kappa_{2}(\hat{L})\kappa_{2}(\hat{U}) \le \frac{a(1-\epsilon)^{\frac{1}{2}}}{23.6dn^{2}\uu \cdot ((1+\epsilon)^{\frac{1}{2}}+t)}. \label{eq:lu1}
\end{equation}
Here, $a$ is a positive constant slightly greater than $1$. Based on \cite[(4.9)]{RCLUPP} and with \eqref{eq:lu1}, the comparison of $\kappa_{2}(X)$ between RCLUPPr and the existing RCLUPP are shown below as
\begin{equation} \label{eq:comparison1}
\kappa_{2}(X)=
\begin{cases}
\displaystyle
a_{1} \cdot \frac{(1+\epsilon)^{\frac{1}{2}}}{(1-\epsilon)^{\frac{1}{2}}}
\cdot \frac{1}{21.4p_{2}}
& \mbox{RCLUPP} \\
\displaystyle
\frac{a(1-\epsilon)^{\frac{1}{2}}}
{23.6dn^{2}\uu \cdot \left((1+\epsilon)^{\frac{1}{2}}+t\right)}
& \mbox{RCLUPPr}
\end{cases},
\end{equation}
with $p_{2} = \max(n^{2}\uu, \frac{k}{(1-\epsilon)^{\frac{1}{2}}-k})$ and $k=\frac{\norm{E_{s}}_{F}}{\norm{X}_{2}}$. $E_{s}$ in \eqref{eq:comparison1} satisfying $\Omega X=\hat{A}+E_{s}$ is the rounding error of matrix sketching in Algorithm~\ref{alg:RCLUPP}. We replace $k_{1}$ in \cite[(15)]{RCLUPP} with $\sqrt{n}$ for comparison. $a_{1}$ is a positive constant. Moreover, the requirement of $\kappa_{2}(\hat{L}))$ between RCLUPPr and LUC2 is shown as
\begin{equation} \label{eq:cl}
\kappa_{2}(\hat{L}) =
\left\{
\begin{array}{ll}
\displaystyle \frac{1}{8\sqrt{mn\uu+n(n+1)\uu}}, & \mbox{LUC2}, \\
\displaystyle \frac{(1-\epsilon)^{\frac{1}{2}}}
{25.5t+25.5csn\sqrt{n}\uu \cdot ((1+\epsilon)^{\frac{1}{2}}+t)}, 
& \mbox{RCLUPPr}.
\end{array}
\right.
\end{equation}

As shown in \eqref{eq:comparison1}, the applicability of RCLUPP is primarily limited by the rounding error in matrix sketching, reflected by $p_{2}$ in its sufficient condition on $\kappa_{2}(X)$. In contrast, \eqref{eq:a} and \eqref{eq:lu1} show that matrix sketching has a negligible effect on RCLUPPr. The key difference is that $\Delta_s$ in \eqref{eq:es} perturbs only the $L$-factor in RCLUPPr, rather than the input matrix $X$. Moreover, comparing \cite[(4.25)]{RCLUPP} with \eqref{eq:pxy-1} shows that RCLUPPr requires a much weaker bound on $\kappa_{2}(\hat{L})\kappa_{2}(\hat{U})$ to control $\kappa_{2}(\hat{W})$. This relaxation is a major theoretical advantage over RCLUPP and RCholeskyQR, which is supported by the numerical results in \cite[Section~6.1.3]{RCLUPP}. Finally, \eqref{eq:cl} highlights the advantage of RCLUPPr over LUC2 by avoiding the potential numerical breakdown of LUC2 when the $L$-factor obtained from LUPP is highly ill-conditioned.

Following the comparison of applicability, we focus on the computational complexity of the preconditioning step before CholeskyQR2 for RCLUPP and RCLUPPr. Similar to \eqref{eq:lu1}, we show the computational complexity of RCLUPP and RCLUPPr for $X \in \mathbb{R}^{m\times n}$ as
\begin{equation} \label{eq:cc}
\mbox{Computational complexity}=
\begin{cases}
\mathcal{O}(smn+mn^{2}) & \mbox{RCLUPP} \\
\mathcal{O}(smn+mn^{2}) & \mbox{RCLUPPr}
\end{cases}.
\end{equation}

Although \eqref{eq:cc} shows that RCLUPP and RCLUPPr have the same asymptotic computational cost, RCLUPPr is less efficient in practice because its dominant operation is LUPP factorization of $X \in \mathbb{R}^{m\times n}$ performed without preconditioning. This is confirmed by the timing results in \cite[Section~6.3]{RCLUPP}. Section~\ref{sec:Numerical} presents several optimization strategies to improve the practical efficiency of RCLUPPr while retaining its advantages in applicability and accuracy.

\section{Experimental results}
\label{sec:Numerical}
For the experimental results, we evaluate the applicability, accuracy, stability, robustness and the efficiency of RCLUPPr through numerical experiments, including several strategies for acceleration. All the numerical experiments are conducted in MATLAB R2025a with the test matrices from the real-world industrial applications and some standard benchmark suites. To assess the applicability and the numerical stability of RCLUPPr, we consider the cases in the single, double and the mixed precision. In the case of the mixed precision, Assumption~\ref{assumption:1} is used with $u_b=2^{-53}$ for the double precision and $u=2^{-24}$ for the single precision. Experiments of the robustness and the efficiency are performed in the double precision. For the randomized algorithms, all the results are averaged over multiple independent runs.

\subsection{Tests of the applicability and numerical stability of RCLUPPr in different precisions}
In this part, we test the applicability and numerical stability of RCLUPPr in different precisions. The comparisons are taken between RCLUPPr, LUC2, RCholeskyQR and RCLUPP. For the numerical experiments, we set $\epsilon=0.5$ and $p=0.6$ as the parameters of matrix sketching in Definition~\ref{definition 2}. We evaluate the numerical stability of the algorithms by measuring both the orthogonality $\norm{\hat{Q}^{\top}\hat{Q}-I_{n}}_{F}$ and the residual $\norm{\hat{Q}\hat{R}-X}_{F}$ across varying $\kappa_{2}(X)$. The influence of $m$, $n$ or other factors can be done in a similar way as that in \cite[Section 6.2]{RCLUPP}.

\subsubsection{Single precision}
In this part, we focus on the case of the single precision with $\uu_{b}=\uu=2^{-24}$ in Assumption~\ref{assumption:1}. We take a numerical example from the problem of distributed optimal control:
\begin{equation}
\begin{aligned}
\min_{\begin{aligned}&u\in L^2(\Omega)\\&a<u<b\end{aligned}}&J(y,u)=\frac{1}{2}\|y-y_d\|^2+\frac{\lambda}{2}\|u\|^2, \\ 
&\text{s.t.}\quad \left\{\begin{aligned}
-\Delta y&=f+u, \quad y \in \Omega \\ \nonumber
y&=0, \quad y \in \partial\Omega
\end{aligned}
\right.
\end{aligned}.
\end{equation}
Here, $f$ and $y_d$ are known data. $\lambda$ is the penalty parameter to guarantee the solvability of this problem. The function $u$ is the control variable, which is subject to a box constraint. To solve the problem of optimal control \cite{1}, the primal-dual active set method, also known as a semi-smooth Newton method, will be applied. It introduces an iteration \cite{3} that solves the matrix $X \in \mathbb{R}^{3n \times 3n}$ repeatedly, which is in the form of
$$
A=\begin{bmatrix}
0 & K & -B \\ 
K & -M & 0 \\ 
EB^T & 0 & I
\end{bmatrix}.
$$
Here, $K \in \mathbb{R}^{n\times n}$ is the stiff matrix. $M \in \mathbb{R}^{n\times n}$ is the mass matrix. $B \in \mathbb{R}^{n\times n}$ is the stiff matrix of the boundary elements. $E \in \mathbb{R}^{n\times n}$ changes with a positive constant $\lambda$ and the active set. When $\lambda$ gets close to $0$, $A$ will become singular. In this group of numerical experiments, we take $n=130$. We construct the testing matrix $X \in \mathbb{R}^{3900\times 390}$ by stacking $10$ $A$ from the top to the bottom, which can guarantee $\kappa_{2}(X) \approx \kappa_{2}(A)$. We fix $s=780$ and vary $\lambda$ from $10^{-3}$, $5\times 10^{-3}$, $10^{-2}$, $5\times 10^{-2}$ to $10^{-1}$. Numerical examples of this group of experiments are shown in Table~\ref{tab:s1} and Table~\ref{tab:s2}.

\begin{table}
\caption{Comparison of Orthogonality in the single precision with different $\kappa_{2}(X)$}
\centering
\begin{tabular}{||c c c c c c||}
\hline
$\kappa_{2}(X)$ & $9.77 \cdot 10^{4}$ & $8.90 \cdot 10^{3}$ & $3.18 \cdot 10^{3}$ & $2.8991 \cdot 10^{2}$ & $1.0287 \cdot 10^{2}$ \\
\hline
LUC2 & $4.75e-6$ & $4.85e-6$ & $4.72e-6$ & $4.86e-6$ & $4.32e-6$ \\
\hline
RHC & $1.02e-5$ & $1.03e-5$ & $1.01e-5$ & $1.02e-5$ & $1.02e-5$ \\
\hline
RCLUPP & $3.40e-6$ & $3.33e-6$ & $3.51e-6$ & $3.48e-6$ & $3.44e-6$ \\
\hline
RCLUPPr & $3.39e-6$ & $3.58e-6$ & $3.55e-6$ & $3.56e-6$ & $3.44e-6$ \\
\hline
\end{tabular}
\label{tab:s1}
\end{table}

\begin{table}
\caption{Comparison of residual in the single precision with different $\kappa_{2}(X)$}
\centering
\begin{tabular}{||c c c c c c||}
\hline
$\kappa_{2}(X)$ & $9.77 \cdot 10^{4}$ & $8.90 \cdot 10^{3}$ & $3.18 \cdot 10^{3}$ & $2.8991 \cdot 10^{2}$ & $1.0287 \cdot 10^{2}$ \\
\hline
LUC2 & $5.30e-3$ & $1.10e-3$ & $5.37e-4$ & $1.05e-4$ & $5.22e-5$ \\
\hline
RHC & $6.00e-3$ & $1.30e-3$ & $6.26e-4$ & $1.19e-4$ & $6.43e-5$ \\
\hline
RCLUPP & $6.20e-3$ & $1.30e-3$ & $6.50e-4$ & $1.19e-4$ & $6.47e-5$ \\
\hline
RCLUPPr & $6.00e-3$ & $1.30e-3$ & $6.39e-4$ & $1.17e-4$ & $6.50e-5$ \\
\hline
\end{tabular}
\label{tab:s2}
\end{table}

Table~\ref{tab:s1} and Table~\ref{tab:s2} show that RCLUPPr achieves numerical stability comparable to that of LUC2 and RCLUPP in both orthogonality and residual in the single precision. Consistent with the findings in \cite[Section 6]{RCLUPP}, RCLUPPr also exhibits superior orthogonality compared with that of RCholeskyQR across different precisions.

\subsubsection{Double precision}
Here, we consider the case of double precision with $\uu_{b}=\uu=2^{-53}$ in Assumption~\ref{assumption:1}. We take a numerical example of the Poisson-type electric-potential equation arising from a one-dimensional capacitively coupled plasma (1D-CCP) fluid simulation \cite{5, 6}. In this example, the ill-conditioned matrix is extracted from the implicit finite-volume discretization of the electric-potential equation along the inter-electrode direction of a 1D-CCP discharge.
The electric-potential equation can be written in the following general form:
\begin{equation}
-\frac{d}{dx}\left(\epsilon_{\mathrm{eff}}\frac{d\phi}{dx}\right)=\rho_{\mathrm{eff}}, \nonumber
\end{equation}
where $\phi$ is the electric potential, and $\epsilon_{\mathrm{eff}}$ and $\rho_{\mathrm{eff}}$ denote the effective permittivity and the effective charge/source term introduced by the implicit coupling between Poisson's equation and the charged-species transport equations. In the plasma region, $\epsilon_{\mathrm{eff}}$ includes the contribution from the implicit electron drift term. $\rho_{\mathrm{eff}}$ contains the ion, electron, and diffusion-related source terms. After the face-based finite-volume discretization, the equation for an interior control volume $P$ can be written as
\begin{equation}
a_P\phi_P+\sum_{N\in\mathcal{N}(P)}a_{PN}\phi_N=b_P, \nonumber
\end{equation}
with
\begin{equation}
a_{PN}=-\frac{S_f}{d_{PN}}\bar{\epsilon}_f, \qquad
\bar{\epsilon}_f=\frac{\epsilon_{\mathrm{eff},P}+\epsilon_{\mathrm{eff},N}}{2}, \qquad
a_P=-\sum_{N\in\mathcal{N}(P)}a_{PN}. \nonumber
\end{equation}
Here, $\mathcal{N}(P)$ denotes the neighboring control volumes of $P$, $S_f$ is the face area, $d_{PN}$ is the distance between the centroids of $P$ and its neighboring control volume $N$, and $\bar{\epsilon}_f$ is the face-interpolated effective permittivity. Therefore, the resulting linear system is
\begin{equation}
A_{1}\phi=b. \nonumber
\end{equation}
In the real 1D-CCP case, we have $A_{1} \in \mathbb{R}^{130\times 130}$ with $\kappa_2(A_{1})=1.21 \times 10^{17}$. Such an $A_{1}$ provides a challenging test case for QR-based linear solvers. To generate the tall skinny $X$ with different $\kappa_{2}(X)$ while preserving the structure of the plasma-discretization matrix, we consider
\begin{equation}
 X_{1}=A_{1}+l_{1}I_{n}. \nonumber
\end{equation}
We vary $l_{1}$ from $10^{-12}$, $10^{-13}$, $10^{-14}$, $10^{-15}$ to $10^{-16}$ to modify $\kappa_2(X_{1})$.
Similar to case of the single precision, we construct the testing matrix $X \in \mathbb{R}^{1300\times 130}$ by stacking $10$ copies of $X_{1}$ vertically with $s=2n=260$. The numerical results for this suite of experiments are summarized in Table~\ref{tab:d1} and Table~\ref{tab:d2}.

\begin{table}
\caption{Comparison of orthogonality in the double precision with different $\kappa_{2}(X)$}
\centering
\begin{tabular}{||c c c c c c||}
\hline
$\kappa_{2}(X)$ & $1.07 \cdot 10^{12}$ & $3.13 \cdot 10^{13}$ & $2.79 \cdot 10^{16}$ & $5.12 \cdot 10^{16}$ & $1.27 \cdot 10^{17}$ \\
\hline
LUC2 & $2.78e-15$ & $2.98e-15$ & $2.99e-15$ & $2.68e-15$ & $3.00e-15$ \\
\hline
RHC & $8.23e-15$ & $8.53e-15$ & $8.79e-15$ & $8.79e-15$ & $8.56e-15$ \\
\hline
RCLUPP & $2.76e-15$ & $2.50e-15$ & $2.45e-15$ & $2.51e-15$ & $2.49e-15$ \\
\hline
RCLUPPr & $2.79e-15$ & $2.74e-15$ & $2.59e-15$ & $2.34e-15$ & $2.88e-15$ \\
\hline
\end{tabular}
\label{tab:d1}
\end{table}

\begin{table}
\caption{Comparison of residual in the double precision with different $\kappa_{2}(X)$}
\centering
\begin{tabular}{||c c c c c c||}
\hline
$\kappa_{2}(X)$ & $1.07 \cdot 10^{12}$ & $3.13 \cdot 10^{13}$ & $2.79 \cdot 10^{16}$ & $5.12 \cdot 10^{16}$ & $1.27 \cdot 10^{17}$ \\
\hline
LUC2 & $4.82e-27$ & $9.93e-16$ & $1.41e-28$ & $9.93e-16$ & $9.93e-16$ \\
\hline
RHC & $2.54e-15$ & $2.22e-15$ & $3.03e-15$ & $2.77e-15$ & $2.69e-15$ \\
\hline
RCLUPP & $2.67e-15$ & $2.93e-15$ & $2.86e-15$ & $2.70e-15$ & $2.65e-15$ \\
\hline
RCLUPPr & $2.31e-15$ & $2.39e-15$ & $2.85e-15$ & $2.74e-15$ & $2.56e-15$ \\
\hline
\end{tabular}
\label{tab:d2}
\end{table}

Table~\ref{tab:d1} and Table~\ref{tab:d2} show that RCLUPPr maintains excellent accuracy and numerical stability for the highly ill-conditioned $X$ in the double precision. According to Table~\ref{tab:s1} and Table~\ref{tab:d2}, we conclude that RCLUPPr is numerically stable in the uniform precision.

\subsubsection{Mixed precision}
\label{sec:mix}
We next examine the numerical stability of RCLUPPr under the setting of the mixed precision as 'single+double' with $\uu=2^{-24}$ and $\uu_{b}=2^{-53}$ in Assumption~\ref{assumption:1}. Following the same strategy as for RCLUPPr, we assign the double precision to all the steps except LUPP factorization up to the computation of the $Q$-factor, $W$, for LUC2, RCholeskyQR and RCLUPP, while performing the remaining steps in the single precision. For RCLUPPr, we focus on two different cases, sRCLUPPr following Assumption~\ref{assumption:1} with LUPP factorization in the single precision and dRCLUPPr with LUPP factorization in the double precision. Two numerical examples are used for the numerical stability and applicability of the algorithms, including the spiked-diagonal matrix and the ill-conditioned triangular matrix. 

The spiked-diagonal matrix arises frequently in control theory and eigenvalue computations, see \cite{Constructing} for further details. Owing to its potentially severe ill-conditioning, it provides a suitable test problem for assessing the robustness of the proposed algorithms. We set $m=20000$, $n=50$ and $s=2n=100$. We let
$v_{1}=(0,1,\ldots,1)^{\top}\in\mathbb{R}^{50},
v_{2}=(1,0,\ldots,0)^{\top}\in\mathbb{R}^{20000}$
and define
$E=\operatorname{diag}\left(1,g^{1/49},\ldots,g^{48/49},g\right) \in \mathbb{R}^{50\times50}$, where $g>0$ is a tunable parameter. We construct
\begin{equation}
X_{1}=
\begin{pmatrix}
E\\
\mathbb{O}
\end{pmatrix}
\in\mathbb{R}^{20000\times50},
\end{equation}
where $\mathbb{O}$ denotes a zero matrix. The test matrix is then defined by
\begin{equation}
X=-5v_{2}v_{1}^{\top}+X_{1}
\in\mathbb{R}^{20000\times50}.
\end{equation}
By choosing $g\in\{10^{-10},10^{-15},10^{-20},10^{-25},10^{-30}\}$, we obtain matrices with different values of $\kappa_{2}(X)$. The corresponding numerical outcomes are shown in Table~\ref{tab:mixa1} and Table~\ref{tab:mixa2}.

We further consider a class of triangular matrices $X \in \mathbb{R}^{m\times n}$ for which the $L$ factor in the LUPP factorization is highly ill-conditioned. Specifically, we construct
\begin{equation}
X=
\begin{bmatrix}
X_{u}+s_{u}I_{n}\\
X_{l}
\end{bmatrix},
\end{equation}
where $X_{u}\in\mathbb{R}^{n\times n}$ is a lower-triangular matrix with unit diagonal entries and $-1$ in all strictly lower-triangular positions. $X_{l}\in\mathbb{R}^{(m-n)\times n}$ is a zero matrix. Such matrices arise in eigenvalue problems, control theory and circuit simulation, see \cite{Elements} and its references for more details. In the experiments, we set $m=4000$, $n=100$ and $s=2n=200$. By varying
$s_{u} \in \{10^{-15},10^{-14},10^{-13},10^{-12},10^{-11}\}$, 
we obtain matrices with different condition numbers $\kappa_{2}(X)$. The numerical results are reported in Table~\ref{tab:mixt1} and Table~\ref{tab:mixt2}.

\begin{table}
\caption{Comparison of orthogonality in the mixed precision for different $\kappa_{2}(X)$ with the spiked-diagonal matrix}
\centering
\begin{tabular}{||c c c c c c||}
\hline
$\kappa_{2}(X)$ & $2.87 \cdot 10^{7}$ & $2.26 \cdot 10^{12}$ & $2.04 \cdot 10^{17}$ & $1.93 \cdot 10^{22}$ & $1.87 \cdot 10^{27}$ \\
\hline
LUC2 & $0$ & $0$ & $0$ & $0$ & $0$ \\
\hline
RHC & $2.17e-6$ & $2.38e-6$ & $3.91e-6$ & -- & -- \\
\hline
RCLUPP & $9.02e-14$ & $1.07e-13$ & $1.04e-13$ & -- & -- \\
\hline
sRCLUPPr & $8.66e-14$ & $8.95e-14$ & $9.17e-14$ & $3.10e-10$ & $8.79e-7$ \\
\hline
dRCLUPPr & $8.27e-14$ & $8.25e-14$ & $7.81e-14$ & $7.34e-14$ & $8.35e-14$ \\
\hline
\end{tabular}
\label{tab:mixa1}
\end{table}

\begin{table}
\caption{Comparison of residual in the mixed precision for different $\kappa_{2}(X)$ with the spiked-diagonal matrix}
\centering
\begin{tabular}{||c c c c c c||}
\hline
$\kappa_{2}(X)$ & $2.87 \cdot 10^{7}$ & $2.26 \cdot 10^{12}$ & $2.04 \cdot 10^{17}$ & $1.93 \cdot 10^{22}$ & $1.87 \cdot 10^{27}$ \\
\hline
LUC2 & $6.62e-8$ & $1.53e-8$ & $4.67e-8$ & $6.72e-9$ & $7.55e-9$ \\
\hline
RHC & $5.18e-6$ & $3.72e-6$ & $3.23e-6$ & -- & -- \\
\hline
RCLUPP & $4.27e-6$ & $4.21e-6$ & $3.96e-6$ & -- & -- \\
\hline
sRCLUPPr & $3.34e-6$ & $1.72e-6$ & $1.72e-6$ & $9.54e-7$ & $1.17e-6$ \\
\hline
dRCLUPPr & $4.62e-6$ & $2.66e-6$ & $1.78e-6$ & $1.50e-6$ & $8.26e-7$ \\
\hline
\end{tabular}
\label{tab:mixa2}
\end{table}

\begin{table}
\caption{Comparison of orthogonality in the mixed precision for different $\kappa_{2}(X)$ with the triangular matrix}
\centering
\begin{tabular}{||c c c c c c||}
\hline
$\kappa_{2}(X)$ & $8.17 \cdot 10^{17}$ & $1.77 \cdot 10^{17}$ & $1.71 \cdot 10^{18}$ & $1.09 \cdot 10^{19}$ & $3.29 \cdot 10^{20}$ \\
\hline
LUC2 & $-$ & $-$ & $-$ & $-$ & $-$ \\
\hline
RHC & $3.15e-4$ & $1.00e-3$ & $3.12e-4$ & $5.45e-4$ & $8.11e-4$ \\
\hline
RCLUPP & $4.39e-6$ & $4.43e-6$ & $4.48e-6$ & $4.58e-6$ & $4.70e-6$ \\
\hline
sRCLUPPr & $4.21e-6$ & $4.14e-6$ & $4.42e-6$ & $4.55e-6$ & $4.25e-6$ \\
\hline
dRCLUPPr & $4.25e-6$ & $4.28e-6$ & $4.52e-6$ & $4.52e-6$ & $4.32e-6$ \\
\hline
\end{tabular}
\label{tab:mixt1}
\end{table}

\begin{table}
\caption{Comparison of residual in the mixed precision for different $\kappa_{2}(X)$ with the triangular matrix}
\centering
\begin{tabular}{||c c c c c c||}
\hline
$\kappa_{2}(X)$ & $8.17 \cdot 10^{17}$ & $1.77 \cdot 10^{17}$ & $1.71 \cdot 10^{18}$ & $1.09 \cdot 10^{19}$ & $3.29 \cdot 10^{20}$ \\
\hline
LUC2 & $-$ & $-$ & $-$ & $-$ & $-$ \\
\hline
RHC & $1.27e-5$ & $1.42e-5$ & $1.40e-5$ & $1.31e-5$ & $1.20e-5$ \\
\hline
RCLUPP & $1.36e-5$ & $1.52e-5$ & $1.84e-5$ & $1.56e-5$ & $1.46e-5$ \\
\hline
sRCLUPPr & $1.51e-5$ & $1.56e-5$ & $1.55e-5$ & $1.63e-5$ & $1.52e-5$ \\
\hline
dRCLUPPr & $1.42e-5$ & $1.45e-5$ & $1.61e-5$ & $1.60e-5$ & $1.57e-5$ \\
\hline
\end{tabular}
\label{tab:mixt2}
\end{table}

Table~\ref{tab:mixa1}-Table~\ref{tab:mixt2} show that RCLUPPr achieves accuracy and numerical stability comparable to those of LUC2 and RCLUPP in the mixed-precision, consistent with the observations in the uniform precision. The advantage of the orthogonality for RCLUPPr over RCholeskyQR is also preserved. For the highly ill-conditioned matrices, Table~\ref{tab:mixa1} and Table~\ref{tab:mixa2} confirm that RCLUPPr retains superior applicability compared with RCholeskyQR. This shows the strong preconditioning effect of performing LUPP factorization before matrix sketching is clearly evidenced by the excellent performance of both LUC2 and RCLUPPr. Table~\ref{tab:mixt1} and Table~\ref{tab:mixt2} demonstrate the superior performance of RCLUPPr over LUC2 when the $L$-factor obtained from LUPP factorization is ill-conditioned, corresponding to the comparison of the theoretical results in \eqref{eq:cl}. Regarding sRCLUPPr with LUPP factorization in the single precision and dRCLUPPr with LUPP factorization in the double-precision, the former frequently encounters numerical breakdown on the arrowhead matrix in the highly ill-conditioned cases. Although slightly slower, dRCLUPPr is considerably more stable and robust owing to the use of LUPP factorization in the double precision LU preconditioning. The single precision offers significantly lower error tolerance than the double precision. Consequently, the rounding error introduced by LUPP is greatly amplified as shown in \eqref{eq:bound}-\eqref{eq:pxy-1}, often leading to numerical breakdown in CholeskyQR2 afterwards. In contrast, the double precision provides a sufficient safety margin, enabling dRCLUPPr to maintain markedly superior robustness in the real cases. In practice, dRCLUPPr can be viewed as a safer and more reliable choice for the real-world experiments. 

\subsection{Tests of robustness of RCLUPPr}
To thoroughly assess the robustness of RCLUPPr, we conduct two groups of experiments in the double precision: the well-conditioned and ill-conditioned cases. As far as we know, this is the first in the existing works of CholeskyQR. For the well-conditioned case, we employ the SVD-based matrix with $m=2000$, $n=50$, and $\norm{X}_{2}=1$. We vary $\kappa_2(X)$ by changing $\sigma_n$ from $10^{-2}$, $10^{-4}$, $10^{-6}$, $10^{-8}$ to $10^{-10}$. We fix $\epsilon=0.5$, $p=0.6$ and $s=2n=100$. We record the times of success and the rate of success for RCLUPPr in $1000$ times, which are shown in Table~\ref{tab:robustness1}. When $\sigma=10^{-10}$, we vary $s$ from $50$, $100$, $200$, $500$ to $1000$ and exhibit the corresponding results in Table~\ref{tab:robustness2}. For the ill-conditioned $X$, we focus on the arrowhead matrices with $m=20000$ and $n=50$. $\kappa_{2}(X)$ is controlled by varying the parameter $g \in\{10^{-15},10^{-20},10^{-25},10^{-30},10^{-35}\}$. We vary $s$ in the same manner as that in the well-conditioned cases. Numerical experiments are summarized in Table~\ref{tab:robustness3} and Table~\ref{tab:robustness4}.

\begin{table}
\caption{Robustness of RCLUPPr in the well-conditioned cases with different $\kappa_{2}(X)$}
\centering
\begin{tabular}{||c c c c c c||}
\hline
$\kappa_{2}(X)$ & $10^{2}$ & $10^{4}$ & $10^{6}$ & $10^{8}$ & $10^{10}$\\
\hline
Times of success & $1000$ & $1000$ & $1000$ & $1000$ & $1000$ \\
\hline
Rates of success & $100\%$ & $100\%$ & $100\%$ & $100\%$ & $100\%$ \\
\hline
\end{tabular}
\label{tab:robustness1}
\end{table}

\begin{table}
\caption{Robustness of RCLUPPr in the well-conditioned cases with different $s$}
\centering
\begin{tabular}{||c c c c c c||}
\hline
$s$ & $50$ & $100$ & $200$ & $500$ & $1000$ \\
\hline
Times of success & $1000$ & $1000$ & $1000$ & $1000$ & $1000$ \\
\hline
Rates of success & $100\%$ & $100\%$ & $100\%$ & $100\%$ & $100\%$ \\
\hline
\end{tabular}
\label{tab:robustness2}
\end{table}

\begin{table}
\caption{Robustness of RCLUPPr in the ill-conditioned cases with different $\kappa_{2}(X)$}
\centering
\begin{tabular}{||c c c c c c||}
\hline
$\kappa_{2}(X)$ & $2.04 \cdot 10^{17}$ & $1.93 \cdot 10^{22}$ & $1.87 \cdot 10^{27}$ & $1.84 \cdot 10^{32}$ & $1.82 \cdot 10^{37}$\\
\hline
Times of success & $1000$ & $1000$ & $773$ & $761$ & $764$ \\
\hline
Rates of success & $100\%$ & $100\%$ & $77.3\%$ & $76.1\%$ & $76.4\%$ \\
\hline
\end{tabular}
\label{tab:robustness3}
\end{table}

\begin{table}
\caption{Robustness of RCLUPPr in the ill-conditioned cases with different $s$}
\centering
\begin{tabular}{||c c c c c c||}
\hline
$s$ & $50$ & $100$ & $200$ & $500$ & $1000$ \\
\hline
Times of success & $779$ & $764$ & $750$ & $766$ & $749$ \\
\hline
Rates of success & $77.9\%$ & $76.4\%$ & $75.0\%$ & $76.6\%$ & $74.9\%$ \\
\hline
\end{tabular}
\label{tab:robustness4}
\end{table}

Table~\ref{tab:robustness1} and Table~\ref{tab:robustness2} demonstrate that RCLUPPr is highly robust in the well-conditioned cases, maintaining numerical stability across thousands of independent trials. For the severely ill-conditioned matrices, Table~\ref{tab:robustness3} shows that numerical breakdown may occur when $\kappa_2(X)$ becomes sufficiently large. Nevertheless, RCLUPPr still achieves a success rate of approximately $75\%$, which substantially exceeds the theoretical lower bound of probability $1-p=0.4$ guaranteed by Definition~\ref{definition 2} and Theorem~\ref{theorem:RCLUPPr}. Table~\ref{tab:robustness4} further indicates that the choice of $s$ for the Gaussian sketch has negligible impact on the robustness in the ill-conditioned regime. Consequently, a small $s=\mathcal{O}(n)$, or even $s=n$, can be safely adopted to improve the efficiency of RCLUPPr without sacrificing the reliability.

\subsection{Accelerating RCLUPPr in the double precision}
In this part, we discuss some practical acceleration strategies for RCLUPPr in the double precision. As shown in \cite[Section~6.3]{RCLUPP}, although RCLUPPr offers superior applicability and accuracy, its computational efficiency lags behind other CholeskyQR-type algorithms. Here, we introduce several targeted optimizations that significantly improve the CPU time (s) while fully preserving the numerical advantages of RCLUPPr.

\subsubsection{Multi-sketching}
For $X \in \mathbb{R}^{m\times n}$, we can use the technique of multi-sketching in \cite{Householder} to replace a single Gaussian sketch. We can use a step of the CountSketch \cite{Low} before the Gaussian sketch to reduce the computational complexity of matrix sketching. When $\Omega_{1} \in \mathbb{R}^{s_{1}\times m}$ is the sketch matrix of the CountSketch with $s_{1} \le m$, according to \cite{frequent, tool}, $s_{1}$ satisfies
\begin{equation}
s \ge \frac{n^{2}+n}{\epsilon^{2}p}. \label{eq:multi}
\end{equation}
With $\Omega \in \mathbb{R}^{s\times s_{1}}$ as the sketch matrix of the Gaussian sketch satisfying $n \le s \le s_{1} \le m$, the comparison between the computational complexity of different types of sketching is shown in Table~\ref{tab:comparisons}.

Table~\ref{tab:comparisons} indicates that multi-sketching becomes cheaper than the single Gaussian sketching when $s(m-s_{1}) \ge m$. For the typical choice $s=\mathcal{O}(n)$ for the Gaussian sketch, \eqref{eq:multi} yields the sufficient condition
\begin{equation}
\frac{n^{2}+n}{\epsilon^{2}p} \le m-\mathcal{O}(\frac{m}{n}), \label{eq:multie}
\end{equation}
which is satisfied whenever $m=an^{2}$ for a positive constant $a$. Accordingly, we retain $\epsilon=0.5$ and $p=0.6$ in the experiments. We generate a test matrix $X\in\mathbb{R}^{20000\times 20}$ using $\mbox{rand}$ in MATLAB and fix $s_{1}$ for the CountSketch as $s_{1}=2800$. Varying $s$ for the Gaussian sketch, we compare the CPU time (s) of RCLUPPr with and without multi-sketching. The results are summarized in Table~\ref{tab:CPU1}. The table clearly demonstrates the effectiveness of multi-sketching in reducing the CPU time (s) once \eqref{eq:multie} still holds. As $s$ increases, RCLUPPr with multi-sketching exhibits a substantial speed advantage over the case with the single Gaussian sketch, which is consistent with the analysis of computational complexity in Table~\ref{tab:comparisons}. The rounding error analysis of RCLUPPr with multi-sketching can be taken following the same lines as in \cite{Householder}.

\begin{table}
\caption{Computational complexity between different sketching for $X \in \mathbb{R}^{m\times n}$}
\centering
\begin{tabular}{||c c||}
\hline
$\mbox{Type of sketching}$ & $\mbox{Computational complexity}$ \\
\hline
$\mbox{Single-sketching}$ & $\mathcal{O}(smn)$ \\
\hline
$\mbox{Multi-sketching}$ & $\mathcal{O}(mn+ss_{1}n)$ \\
\hline
\end{tabular}
\label{tab:comparisons}
\end{table}

\begin{table}
\caption{Comparison of the CPU time (s) of RCLUPPr with different  matrix sketching}
\centering
\begin{tabular}{||c c c c c c||}
\hline
$s$ & $100$ & $200$ & $400$ & $800$ & $1600$ \\
\hline
Single-sketching & $0.0432$ & $0.0456$ & $0.0492$ & $0.0593$ & $0.0802$ \\
\hline
Multi-sketching & $0.0425$ & $0.0431$ & $0.0453$ & $0.0455$ & $0.0463$ \\
\hline
Reduction of the CPU time & $1.7\%$ & $5.8\%$ & $8.6\%$ & $30.3\%$ & $42.3\%$ \\
\hline
\end{tabular}
\label{tab:CPU1}
\end{table}

\subsubsection{Adaptive RCLUPPr}
We also investigate practical variants of RCLUPPr to further improve the efficiency. RCLUPPr with multi-sketching remains applicable for $X\in\mathbb{R}^{m\times n}$ whenever $m=\mathcal{O}(n^2)$, see \eqref{eq:multie}. Here, we consider performing LU factorization without partial pivoting, which is significantly faster in many cases. Replacing the line $[L,U,P]=\mathrm{LU}(X)$ with $[L,U]=\mathrm{LU}(X)$ in Algorithm~\ref{alg:RCLUPPr} yields the variant RCLUr, see Algorithm~\ref{alg:RCLUr}. To combine the high robustness of RCLUPPr with the speed of RCLUr, we propose an adaptive strategy called Adaptive RCLUPPr in Algorithm~\ref{alg:adaptive}). It attempts the faster RCLUr in the beginning. Only upon detecting numerical breakdown (NaN/Inf) or when the tolerances of the orthogonality and residual are violated does it fall back to RCLUPPr. This hybrid approach achieves near-optimal runtime for the well-conditioned matrices while automatically guaranteeing the robustness for the ill-conditioned cases without requiring prior knowledge of the matrix properties. 

For the numerical experiments, we use the SVD-based matrix with $m=2000$, $n=50$, $\norm{X}_{2}=1$ and $\kappa_{2}(X)=10^{16}$. We set $\tau_{ortho}=\tau_{res}=10^{-15}$. Table~\ref{tab:CPU2} compares the CPU time (s) of RCLUPPr and Adaptive RCLUPPr in this example. We find that Adaptive RCLUPPr is substantially faster than RCLUPPr, illustrating that it is a practical and efficient option for many ill-conditioned problems.

\begin{algorithm}
\caption{$[Q,R]=\mbox{RCLUr}(X)$}
\label{alg:RCLUr}
\begin{algorithmic}[1]
\STATE $[L,U]=\mbox{LU}(X),$
\STATE The remaining steps are the same as steps $2-7$ of RCLUPPr in Algorithm~\ref{alg:RCLUPPr}.
\end{algorithmic}
\end{algorithm}%

\begin{algorithm}
\caption{Adaptive RCLUPPr}
\label{alg:adaptive}
\begin{algorithmic}[1]
\STATE \textbf{Phase 1: run RCLUr without partial pivoting}
\STATE Compute $\text{ortho} \gets \norm{Q^{\top}Q-I_{n}}_{F}$ and $\text{res} \gets \norm{QR-X}_{F}$
\IF{any entry in $Q$ or $R$ is NaN or Inf \textbf{or} $\text{ortho} > \tau_{ortho}$ \textbf{or} $\text{res} > \tau_{res} \cdot \norm{X}_{F}$}
    \STATE \textbf{Phase 2: run RCLUPPr with partial pivoting}
\ENDIF
\end{algorithmic}
\end{algorithm}

\begin{table}
\caption{Comparison of the CPU time (s) between RCLUPPr and Adaptive RCLUPPr}
\centering
\begin{tabular}{||c c c c c c||}
\hline
$s$ & $50$ & $100$ & $200$ & $400$ & $800$\\
\hline
RCLUPPr & $0.0065$ & $0.0074$ & $0.0084$ & $0.0093$ & $0.0112$\\
\hline
Adaptive RCLUPPr & $0.0032$ & $0.0038$ & $0.0048$ & $0.0059$ & $0.0063$ \\
\hline
Reduction of the CPU time & $50.8\%$ & $48.6\%$ & $42.9\%$ & $36.7\%$ & $43.8\%$ \\
\hline
\end{tabular}
\label{tab:CPU2}
\end{table}

\section{Conclusions}
\label{sec:conclusions}
This work provides the first comprehensive rounding error analysis of RCLUPPr, a novel randomized CholeskyQR-type algorithm with direct LUPP factorization. We rigorously prove that performing LUPP factorization before matrix sketching yields better applicability compared with the existing algorithms. The algorithm is also shown to be numerically stable and accurate through detailed rounding error analysis. We further demonstrate its robustness through extensive experiments on the matrices from the real-world problems. Some practical strategies for accelerating RCLUPPr, including multi-sketching and Adaptive RCLUPPr based on LUPP factorization, significantly improve the CPU time (s) while preserving all the theoretical advantages in the real implementations. Numerical results on the matrices from the optimal control, capacitively coupled plasma simulations, and the control theory confirm the effectiveness and practicality of RCLUPPr.

\section*{Acknowledgement}
\label{sec:acknowledgement}
The authors declare that they have no conflict of interest. Haoran Guan and Zhenyu Zou are the co-first authors of this article with the dagger symbol (\textsuperscript{\textdagger}).‌‌

\bibliography{references}

\end{document}